\documentclass[10pt]{amsart}
\usepackage[utf8]{inputenc}
\usepackage[english]{babel}
\usepackage[T1]{fontenc}
\usepackage{indentfirst,microtype}
\usepackage[margin=2.5cm]{geometry}
\usepackage{amsfonts,amsthm,amsmath,amssymb,amsrefs,amscd,cite}
\usepackage{enumerate}
\usepackage[colorlinks=true,linkcolor=blue,citecolor=magenta]{hyperref}
\usepackage{eurosym}

\newtheorem{thm}{Theorem}

\newtheorem*{question}{Question}
\newtheorem{proposition}[thm]{Proposition}

\newtheorem{lemma}[thm]{Lemma}
\newtheorem{definition}{Definition}
\theoremstyle{remark}

\newtheorem{remark}[thm]{Remark}

\title{Local rigidity for periodic generalised interval exchange transformations}

\author[Selim Ghazouani]{Selim Ghazouani}
\address{Mathematics Institute, University of Warwick, Coventry CV4 7AL, U.K. }
\email{s.ghazouani@warwick.ac.uk}

\begin{document}

\maketitle

\begin{abstract}
In this article we study local rigidity properties of generalised interval exchange maps using renormalisation methods. We study the dynamics of the renormalisation operator $\mathcal{R}$ acting on  the space of $\mathcal{C}^{3}$-generalised interval exchange transformations at fixed points (which are standard periodic type IETs). We show that $\mathcal{R}$ is hyperbolic and that the number of unstable direction is exactly that predicted by the ergodic theory of IETs and the work of Forni and Marmi-Moussa-Yoccoz. As a consequence we prove that the local $\mathcal{C}^1$-conjugacy class of a periodic interval exchange transformation, with $d$ intervals, whose associated surface has genus $g$ and whose Lyapounoff exponents are all non zero is a codimension $g-1 +d-1$ $\mathcal{C}^1$-submanifold of the space of $\mathcal{C}^{3}$-generalised interval exchange transformations. This solves a particular case of a conjecture of Marmi-Moussa-Yoccoz.

\end{abstract}

\section{Introduction}

The study of stability and rigidity properties of quasi-periodic and parabolic dynamical systems form a rather old class of problems in the modern theory of dynamical systems. Trying to determine whether the solar system is stable led astronomers to formulate simplified mathematical problems, one of which  being the famous \textit{three-body problem}. Daunted by the many difficulties arising when trying to solve it, Poincaré \cite{Poincare} suggested that mathematicians turn to even simpler toy models, such as the dynamics of circle diffeomorphisms.

\noindent The \textit{three-body problem} was eventually solved by Kolmogorov in 1954 using a set of methods nowadays commonly referred to as KAM theory. These methods were subsequently applied by Arnol'd \cite{Arnold} to solve the problem of the local rigidity for analytic circle diffeomorphisms. 

\noindent Later on, from the late 1970s onwards, the introduction of renormalisation as a tool in mathematics allowed mathematicians to discover a few more rigidity and universality phenomena for other classes of parabolic dynamical systems, in particular unimodal maps \cite{CoulletTresser,Feigenbaum,Sullivan,McMullen,Lyubich,AvilaLyubichdeMelo}, circle diffeomorphisms with critical points \cite{deFariadeMelo,deFariadeMelo2} or breaks points \cite{K1,KhaninKocicMazzeo} and more recently circle maps with a flat interval \cite{MartensPalmisano}. The general question that can be asked at this point is the following

 \begin{question}
What classes of dynamical systems are rigid (in some sense) and what are the mechanisms responsible for this rigidity?
 \end{question}
 
The notion of rigidity is vague and can mean different things depending on the context. We give here rigidity themes that we have in mind when writing this text and which are interconnected.

\begin{itemize}

\item \textit{Geometric rigidity.} This is when the topological structure of a dynamical system forces its geometry. Formally, we say a smooth dynamical system is  \textit{geometrically rigid} if any other system which is topologically conjugate to it is differentiably conjugate to it. This is the case for certain classes of circle maps and infinitely renormalisable unimodal maps.

\item \textit{Universality.} We say a class of dynamical systems displays some form of universality if some universal behaviour can be observed in arbitrary parameter families. An example is parameter families of unimodal maps displaying period-doubling bifurcations and for which the structure of bifurcations asymptotically does not depend on the initial family.

\item \textit{Solving cohomological equations.} Analysing local geometric rigidity problems via linearising the problem often features solving \textit{cohomological equations}. Understanding the obstructions to solving cohomological equations is an important step to solving rigidity problems, and describing distributions realising these obstructions an interesting problem.

\end{itemize}  
 
\noindent Most of known results for geometric rigidity and universality are either in dimension $1$ or are local results where the underlying system is a translation on a torus and KAM theory applies. 
 \vspace{2mm}
 
 \paragraph*{\bf Ergodic theory} A lot geometric rigidity results rely partly on the fact that the combinatorial structure (and consequently the ergodic theory) of underlying dynamical systems is simple; precisely it is either a translation on a torus or an odometer. It is the case for results in KAM theory, circle diffeomorphisms, circle diffeomorphisms with critical or break points and unimodal maps. For instance, KAM theory relies on Fourier analysis to solve the cohomological equation over rotations of the $n$-dimensional torus. This is only made possible because $n$-dimensional tori are abelian groups, translations preserve this group structure which makes for an efficient Fourier analysis. More general parabolic systems are more complicated and we do not always have ready tools for a direct analysis of their ergodic theory, study of the cohomological equation and deviations of ergodic averages. 
 
There has been important progress in understanding the ergodic theory of parabolic dynamical systems: \cite{Zorich1}, \cite{KonsevichZorich} and \cite{Forni} for flows on surfaces, \cite{FlaminioForni,FlaminioForni2} for the horocycle flow and nilflows and \cite{MMY1,MarmiYoccoz} for interval exchange maps. For all these examples it is shown that deviations of ergodic average for smooth observable are governed by finitely many distributions, and for functions in the kernel of those distributions one can solve the cohomological equation. These distributions also provide finitely many obstructions to geometric rigidity.

\vspace{3mm}

\paragraph*{\bf Main result} In this article we prove a local (geometric) rigidity result for generalised interval exchange transformations. A generalised interval exchange transformation (GIET) is a bijection of the interval which is piecewise continuous, smooth and increasing on its continuity intervals (see Section \ref{giet} for precise definitions). These maps, which are obtained as first-return of smooth flows on surfaces, have vanishing entropy and are most of the time uniquely ergodic\footnote{In parameter space, one expects the generic GIET to be Morse-Smale, but interesting cases are infinitely renormalisable maps; their combinatorial structure can be reduced to that of standard IETs which are know to almost always be uniquely ergodic, see \cite{YoccozNotes}.}.  Our main result is

\begin{thm}
\label{maintheorm}
Let $T_0$ be a periodic IET with hyperbolic Rauzy matrix. The set of generalised IETs which are $\mathcal{C}^{1}$-conjugated to $T_0$ by a diffeomorphism $\mathcal{C}^1$-close to the identity is a  $\mathcal{C}^1$-submanifold of codimension $d-1 +g-1$. 
\noindent Here $d$ is the number of intervals of $T_0$ and $g$ is the genus of the associated translation surface, and we are working in the space of $\mathcal{C}^3$-generalised IETs whose total non-linearity vanishes.
\end{thm} 

\noindent This result (and more) had been conjectured by Marmi, Moussa and Yoccoz in \cite{MMY}. Their main conjecture (Problem 1 in \cite{MMY}) is that this result is true for almost every choice of initial IET $T_0$; here we only treat the case of periodic combinatorics. This result is nonetheless (to the best knowledge of the author) the first result describing a rigidity class of generalised interval exchange transformations.

\noindent We briefly comment on the statement of this theorem which might seem a bit technical at first glance. Consider $T_0$ a standard interval exchange transformation. We can deform it within the Banach space of $\mathcal{C}^3$-generalised interval exchange transformation and ask whether the new map is differentiably conjugate to $T_0$. 

\begin{itemize}
\item  A necessary condition for this to happen is that they are orbitally equivalent; this condition is realised if and only their \textit{generalised rotation number/Rauzy path} (see \cite{YoccozNotes}) is the same. Roughly, this generalised rotation number takes value in a $(d-1)$-dimensional simplex and provides a first set of combinatorial obstructions.

\item Once we know that this first condition is satisfied, the two maps that we get have same ergodic theory (as it can be shown that they are semi-conjugate). In this case, the ergodic theory (via the work of Forni \cite{Forni,Forni2} and Marmi-Moussa-Yoccoz \cite{MMY1}) provides us with a new set of obstructions which correspond to obstructions to solving the cohomological equation. There are $g-1$ such obstructions, and correspond to Lyapounov exponents of the Konsevich-Zorich cocyle.

\end{itemize}

\noindent This result shows that these obstructions are the only obstructions to local rigidity. It is very much in the spirit of standard rigidity results about circle maps or unimodal maps, once the ergodic theory has been factored in.

\vspace{3mm}

\paragraph*{\bf Renormalisation}

The proof of Theorem \ref{maintheorm} makes use of renormalisation methods. A powerful idea to study parabolic dynamical systems is to consider a renormalisation operator acting on the moduli space of such systems. A renormalisation operator is a procedure by which one associates to a dynamical system $T$ a suitably rescaled first-return map, that we denote by $\mathcal{R}(T)$, which is in the same class as $T$ (in our case, a generalised IET with as many discontinuity intervals).  

A general principle is that two maps $T_1$ and $T_2$ are differentiably conjugate if and only if their iterated renormalisations $\mathcal{R}^n(T_1)$ and $\mathcal{R}^n(T_2)$ are getting close at an exponential rate. In this article we make use of a renormalisation operator $\mathcal{R}$ on generalised IETs (which is an extension of standard Rauzy-Veech induction). Our main theorem is just a corollary of the following result

\begin{thm}
\label{thm2}
Let $T_0$ be a standard IET which is a fixed point of $\mathcal{R}$ and whose standard Rauzy matrix is hyperbolic. Then $\mathcal{R}$  acting on the Banach manifold of $\mathcal{C}^3$-generalised IETs is  hyperbolic at $T_0$, its unstable space has dimension $(d-1) + (g-1)$ and consequently its stable space has codimension  $(d-1) + (g-1)$.
\end{thm}

\vspace{3mm}

\paragraph*{\bf Previous work on generalised interval exchange transformations} We discuss briefly previous work on the question on generalised interval exchange transformations of genus $2$. As already mentioned, pioneering work of Forni followed by Marmi-Moussa-Yoccoz and Marmi-Yoccoz on the solving of the cohomological equation set the stage for a discussion on rigidity properties of generalised interval exchange maps. They demonstrated the existence of obstructions to solving the cohomological equation and interpreted them as cohomology classes of the associated surface (obtained by suspending a generalised interval exchange map).

\vspace{2mm}

\noindent Subsequent work of Marmi-Moussa-Yoccoz implemented a KAM scheme to describe local smooth conjugacy classes of generalised IETs in high regularity ($\mathcal{C}^r$-conjugacy classes for $r \geq 2$). They give a formula for the codimension of such conjugacy classes in terms of $d$ and $g$. Their result cover almost every rotation number, but they fail to describe $\mathcal{C}^1$-rigidity classes which are the generic case in parameter space. Their work was completed by Forni-Marmi-Matheus \cite{ForniMarmiMatheus} to cover other rotation numbers, still in high regularity.

\vspace{3mm}

\paragraph*{\bf Strategy of the proof} We comment on the proof of Theorem \ref{maintheorm} and Theorem \ref{thm2} (we assume some knowledge of renormalisation theory). We consider the renormalisation operator acting at a fixed point $T_0$.

\begin{enumerate}

\item We first show the existence of $(d-1) + (g-1)$ unstable directions by letting $\mathcal{R}$ act on the finite dimensional subspace of \textit{affine} interval exchange transformations. This action can be related to the standard action of the Zorich-Konsevich cocycle and we can achieve our aim by standard ergodic-theoretic methods.

\item The difficult part of the problem is to construct the stable space. Indeed the complement of the $(d-1) + (g-1)$ unstable directions is infinite dimensional and we have a priori very little control on what happens there. In many standard cases deriving from circle maps (circle diffeomorphisms, circle maps with critical points or break points), a strong control is given by what is nothing short of an ergodic miracle, the Denjoy-Koksma inequality. It provides what specialists call \textit{a priori bounds} for the renormalisation. 

\item We construct a \textit{pre-stable} space of codimension $(d-1) + (g-1)$ satisfying the property that for any $T$ in this pre-stable space, the sequence $\big( \mathcal{R}^n(T) \big)$ is bounded in the $\mathcal{C}^2$-topology. We see this latter statement as an a priori bound. This construction is the heart of the article. It makes use of the fact that $\mathcal{R}$ is hyperbolic restricted to the subspace of affine IETs, various distortion bounds for one-dimensional dynamical systems and the choice of an appropriate norm using the non-linearity of one-dimensional maps. A key idea of this construction is to carry out corrections to shadow the sequence $\mathcal{R}^n(T)$ by renormalisations of affine IETs. This was inspired by the proof of the main theorem of \cite{MMY1}.

\item Once those "a priori bounds" are constructed, we obtain uniform contraction in the pre-stable space as a reformulation of Herman's theory for circle diffeomorphisms. 

\item We derive the main theorem using standard result for renormalisation of one-dimensional maps borrowed from \cite{deFariadeMelo}.

\item The regularity of the stable space is obtained by writing the equation defining it and using a result of Marmi-Yoccoz \cite{MarmiYoccoz} on the regularity of solutions of the cohomological equation.

\end{enumerate}

\vspace{2mm}

\paragraph{\bf Acknowledgements}  The author would like to thank Liviana Palmisano for sharing course notes about renormalisation, Michael Bromberg and Björn Winckler for interesting discussions, and Giovanni Forni his careful reading and precious comments on an early versions of this text. The author is greatly indebted to Corinna Ulcigrai for sparking his interest in the subject, her teaching, the many hours of conversation about interval exchange maps and her continued support.

\newpage

\section{Generalised interval exchange transformations}
\label{giet}

\subsection{Basic definitions}

\begin{definition}
Let $d\geq 2$ be an integer. A $\mathcal{C}^r$-generalised interval exchange transformation (GIET) is a map $T$ from the interval $[0,1]$ to itself such that

\begin{itemize}

\item there are two partitions $[0,1] = \bigcup_{i=1}^d{I^t_i} =  \bigcup_{i=1}^d{I^b_i} $ into  $d$ open subinterval (the intervals $I_i^t$s and $I_i^b$s are lying on $[0,1]$ ordered from left to right);

\item  there exists a permutation $\sigma \in \mathfrak{S}_n$ such that $T$ restricted to $I_i^t$ is an orientation preserving diffeomorphism onto $I_{\sigma(i)}^b$ of class $\mathcal{C}^r$;

\item $T$ extends to the closure of $I_i^t$ to a $\mathcal{C}^r$-diffeomorphism onto the closure of $I_{\sigma(i)}^b$.

\end{itemize}

\end{definition}

\noindent Examples of such generalised interval exchange transformations include \textit{standard} interval exchange transformations (IET) for which the map $T$ is further restricted to be a translation on each of the $I_i^t$s and \textit{affine} interval exchange transformations (AIET) for which $T$ is an affine map restricted to the $I_i^t$s.

\vspace{3mm}

\noindent In what follows we make the standing assumption that $r \geq 2$. Let $T$ be a $\mathcal{C}^r$-GIET. We define 

\begin{equation}
\eta_T = \mathrm{D} \log \mathrm{D}T
\end{equation}

\noindent which is called the \textit{non-linearity} of $T$ and is well-defined because we have assumed $T$ is $\mathcal{C}^2$. 

\noindent If $f : I \longrightarrow J$ is a continuous function from a bounded interval $I$ to another $J$, we use the following notation

$$ ||f|| = ||f||_0 = \sup_{x\in I}{|f(x)|}.$$

\subsection{The moduli space and coordinates}
\label{coordinates}
\noindent We define 

$$ \mathcal{X}_{\sigma}^r = \{  \text{generalised interval exchange transformation of class} \ \mathcal{C}^r \ \text{with associated permation} \ \sigma  \}.$$

\noindent Let $T$ be a $\mathcal{C}^{r}$-GIET, with associated permutation $\sigma$ and let $(I_i^t)_{1 \leq i\leq d}$ and $(I_i^b)_{1 \leq i\leq d}$ be the "top" and "bottom" partitions of $[0,1]$ associated to it. We make the two following observations.

\begin{itemize}
\item  There is  a unique affine interval exchange transformation $A_T$ mapping $I_i^t$ to $I_{\sigma(i)}^b$. 

\item Furthermore, for all $i\leq d$, there is a unique element $\varphi_T^i$ of $\mathrm{Diff}^r([0,1])$  such that the restriction of $T$ to $I_i^t$ is equal to 

$$ c_i \circ  \varphi_i \circ b_i $$ where $b_i$ is the unique orientation preserving affine map mapping $I_i^t$ onto $[0,1]$ and  $c_i$ is the unique orientation preserving affine map mapping $[0,1]$ onto $I_{\sigma(i)}^b$. 

\end{itemize}

\noindent This operation can be inverted and therefore the map

$$ T \longmapsto (A_T, \varphi_T^1, \cdots, \varphi_T^d) $$ gives an identification between $\mathcal{X}_{\sigma}^r$ and $\mathcal{A}_{\sigma} \times \big(\mathrm{Diff}^r([0,1])\big)^d$ where $\mathcal{A}_{\sigma}$ the space of AIETs with permutation $\sigma$. In the sequel we denote by $\mathcal{P}$ the space $\big(\mathrm{Diff}^r([0,1])\big)^d$ and so we have a canonical identification 

$$ \mathcal{X}_{\sigma}^r  = \mathcal{A}_{\sigma} \times \mathcal{P}.$$ Using this parametrisation we can endow $\mathcal{X}_{\sigma}^r$ with the structure of a Banach space directly inherited from that of  $\mathrm{Diff}^r([0,1])$. When there is no possible ambiguity, we will drop the indexes $\sigma$ and $r$ and simply write 

$$ \mathcal{X} = \mathcal{A} \times \mathcal{P}.$$

\subsection{Renormalisation}

\noindent We introduce in this paragraph a map acting upon $ \mathcal{X}_{\sigma}^r $ which is a \textit{renormalisation operator}. A fully-fledged renormalisation theory for GIETs would require that we introduce the \textit{Rauzy-Veech induction}, as it is done in \cite{MMY}. However, because we are only going to treat a particular combinatorial case, we can spare such machinery and define everything in more elementary terms.

\vspace{3mm} \noindent In the sequel $T_0$ is a standard IET which satisfies the following self-similarity property: there exists $x_0 \in ]0,1[$ such that the first-return map of $T_0$ on $[0,x_0]$ is equal, up to affine rescaling,  to $T_0$. Consequently, there is a neighbourhood $\mathcal{W}$ of $T_0$ in $\mathcal{X}$ and a smooth map $X : \mathcal{X} \longrightarrow [0,1[$  such that the following holds.

\begin{itemize}

\item $X(T_0) = x_0$;

\item For every $T \in \mathcal{W}$, the first return map of $T$ on $[0,X(T)]$ is a GIET with permutation $\sigma$;

\item if $\mathcal{R}T$ denotes this first return map rescaled to define a function from $[0,1]$ to itself, the map 

$$ \mathcal{R} : \mathcal{W} \longrightarrow \mathcal{X} $$ is continuous;

\item if we denote by $\mathcal{R}_{\mathcal{A}}$ and $\mathcal{R}_{\mathcal{P}}$ the projection of $\mathcal{R}$ on the coordinates $\mathcal{A}$ and $\mathcal{P}$ respectively the map  

$$ \mathcal{R}_{\mathcal{A}} : \mathcal{W} \longrightarrow \mathcal{A} $$ is of class $\mathcal{C}^1$;

\item $ \mathcal{R}(T_0) = T_0$;

\item for all $T \in \mathcal{X}$, $ \mathrm{D}\mathcal{R}_{\mathcal{A}}(T)$ is a bounded operator for the $\mathcal{C}^r$-norm. 
\end{itemize}

\noindent The facts that  $\mathcal{R}$ is of class $\mathcal{C}^1$ and  $ \mathrm{D}\mathcal{R}_{\mathcal{A}}(T)$ is a bounded operator are a  consequence of the fact that $\mathcal{R}(T)$ is obtained by taking compositions of the restrictions of $\mathcal{T}$ to its continuity intervals on intervals whose endpoints themselves depend smoothly on $\mathcal{T}$ (the proof of these facts is discussed in greater detail in Appendix \ref{appendix}). In the sequel we will be calling $\mathcal{R}$ the \textit{renormalisation operator}. For a given GIET $T \in \mathcal{W}$, we will call $\mathcal{R}T = \mathcal{R}(T)$ its \textit{renormalisation} and when well-defined, we call the sequence $T,   \mathcal{R}T, \mathcal{R}^2T, \cdots, \mathcal{R}^nT, \cdots$ its \textit{consecutive renormalisations}. When it is the case that consecutive renormalisation of $T$ are defined for all $ n \geq 0$, \textit{i.e.} $\mathcal{R}^n T \in \mathcal{W}$ for all $n\geq 0$, we say that $T$ is infinitely renormalisable.

\begin{remark}

The reason why we care about such a renormalisation operator is the following: a GIET in $\mathcal{W}$ is $\mathcal{C}^1$-conjugate to $T_0$ if and only if its consecutive renormalisations converge fast enough to $T_0$. This rather loose statement will be made precise in Section \ref{rigidity}.

\end{remark}

\subsection{Dynamical partitions}

Let $T$ be an element of $\mathcal{W}$ and assume further that $T$ is infinitely renormalisable. For any $n \geq 0$, $\mathcal{R}^nT$ is the rescaling of a first return map of $T$ on an interval of the form $[0,x_n]$. The interval $[0,x_n]$ is partitioned into 

$$ [0,x_n] = \cup_{j=1}^d{I^j_n} $$ and $\mathcal{R}^nT$ rescaled down to $[0,x_n]$ is equal to $T^{l^j_n}$ on each of the $I^j_n$s. For $1 \leq j \leq d$, we introduce

$$ \mathcal{P}^j_n = \{ I^j_n, T(I^j_n),  T^2(I^j_n), \cdots, T^{l^j_n-1}(I^j_n)  \} $$ and we call 

$$ \mathcal{P}_n = \bigcup_{j=1}^d{\mathcal{P}^j_n} $$ the dynamical partition of level $n$. One easily verifies that $\mathcal{P}_n$ is a partition of $[0,1]$ into subintervals.

\section{Affine interval exchange transformations}
\label{affine}

\noindent An \textit{affine interval exchange transformation} is simply a generalised IET which is affine restricted to its intervals of continuity. In this subsection, we aim at computing the derivative of the renormalisation operator restricted to AIETs, at the fixed point $T_0$. We will see that this derivative can be understood fairly simply in terms of the combinatorial structure of $T_0$.

\subsection{Coordinates on $\mathcal{A}$.} Let $T$ be an AIET with permutation $\sigma$. Denote by $\lambda_1, \cdots, \lambda_d$ the lengths of its continuity intervals. Because these form a partition of $[0,1]$, they must satisfy the following equation:

$$ \lambda_1 + \cdots + \lambda_d = 1.$$
 
\noindent Furthermore, if we denote by $\rho_1, \cdots, \rho_d$ the derivatives of $T$ on intervals of respective lengths $\lambda_1, \cdots, \lambda_d$, we must also have

$$  \rho_1\lambda_1 + \cdots + \rho_d\lambda_d = 1.$$ These two equations, together with the further restrictions that $\forall i, \lambda_i > 0$ identify $\mathcal{A}_{\sigma}$ to a submanifold of $\mathbb{R}^{2d}$ of dimension $2d-2$. For any affine interval exchange transformation $T$, we denote by $\lambda(T)$ its associated lengths and $\rho(T)$ its slopes.

\vspace{2mm}

\paragraph{\bf Surface associated to an IET} To an IET can be associated a a topological surface with marked points by an operation of \textit{suspension}. If $s$ is the number of marked points of this surface and $g$ its genus we have the following relation

$$ d = 2g +s -1.$$ 

\noindent We make the standing assumption that $s$ \textbf{is equal to} $1$ \textbf{and that} $g \geq 2$.

\subsection{Intersection matrix}
\label{intersectionmatrix}

Recall $T_0$ the fixed point of $\mathcal{R}$ and the $\mathcal{P}^j_n$s the sub-partitions associated with the dynamical partitions $\mathcal{P}_n$. Define $a_{ij}$ to be the number of elements of $\mathcal{P}^i_n$ which intersect $I^j_0$. The $I^j_0$ are just by definition the intervals of continuity of $T_0$. We will denote by $A$ the $d\times d = 2g \times 2g$ matrix whose entry in place $(i,j)$ is $a_{ij}$. We call $A$ the \textit{intersection matrix} of $A$. We have the following well-known facts about $A$ (we refer to \cite{Yoccoz1} for details and proofs).

\begin{enumerate}

\item All coefficients of $A$ are positive(possibly requires passing to a power of $\mathcal{R})$.

\item $(\lambda_1^0, \cdots, \lambda_d^0)$  the lengths of $T_0$ is an eigenvector of $^tA$. 

\item The associated eigenvalue is simple and is the the largest eigenvalue of $^tA$.

\item $A$ preserves a (non-degenerate) symplectic form.

\end{enumerate}

\noindent

 We want to understand the action of $\mathcal{R}$ on $\mathcal{A}$ close to $T_0$. Note that $\mathcal{R}$ stabilises the subset of standard IETs (defined in coordinates by $\rho_1 = \cdots = \rho_d = 1$). This subset identifies with the simplex $ \Delta = \{ (\lambda_1, \cdots, \lambda_d) \in \mathbb{R}_+ \ | \ \sum{\lambda_i} = 1 \}$ and the action of $\mathcal{R}$ restricted to it is nothing but the projective action of $(^tA)^{-1}$. From all these considerations we get the following fact:

$$ T_0 \ \text{is an expanding fixed point of} \ \mathcal{R} \ \text{restricted to IETs.} $$ By that we mean that $(\mathrm{D}R)_{T_0}$ the derivative of $\mathcal{R}$ satisfies for all $v \in T_{T_0}\Delta$, $||(\mathrm{D}R)_{T_0} v|| > \alpha ||v||$ for a certain norm $|| \cdot ||$ and $\alpha>1$. 

\vspace{2mm}

\noindent Another important fact is that the action of $\mathcal{R}$ on the slopes $\rho = (\rho_1, \cdots, \rho_d)$ satisfies the following: if $\mu(T) = \log \rho(T) = \big(\log \rho_1(T), \cdots, \log \rho_d(T) \big)$ we have 

$$ \mu(\mathcal{R}T) = A \cdot \mu(T).$$ irrespective of the value of $\lambda(T)$.

\subsection{Derivative of $\mathcal{R}$ restricted to $\mathcal{A}$ at $T_0$.}

We make the following standing assumption for the rest of the article: 

$$ A \ \textbf{is a hyperbolic matrix}.$$ Because $A$ preserves a symplectic form, it has $g$ eigenvalues which are (strictly) larger than $1$ and $g$ which are (strictly) smaller than $1$. We briefly discuss how little restrictive this assumption in \ref{assumption}. 

\noindent Using coordinates $(\lambda, \mu)$ introduced above, we write $\mathcal{R} = (\mathcal{R}_{\lambda}, \mathcal{R}_{\mu})$.

\begin{proposition} The following statements hold true:

\begin{enumerate}

\item $(\mathrm{D}_{\lambda}\mathcal{R}_{\mu})_{T_0} = 0$;

\item there exists $\alpha>1$ such that $(\mathrm{D}_{\lambda}\mathcal{R}_{\lambda})_{T_0}$ is $\alpha$-expanding;

\item $(\mathrm{D}_{\mu}\mathcal{R}_{\mu})_{T_0}$ is hyperbolic and has $g-1$ expanding directions.

\end{enumerate}

\end{proposition}

\begin{proof}

The proof is straightforward. A neighbourhood of $T_0$ can be parametrised using coordinates $(\lambda,\mu)$ has above. If $\lambda^0$ is the coordinates associated to $T_0$, the tangent space of $\mathcal{X}$ at $T_0$ is defined by the following equations 

$$ \sum{\lambda^i_0} =0 $$ and $$ \sum{\mu_i\lambda^i_0} = 0.$$ 

\begin{enumerate}
\item $(\mathrm{D}_{\lambda}\mathcal{R}_{\mu})_{T_0} = 0$ is a simple consequence of the fact that the space of linear IETs is stable by $\mathcal{R}$;

\item as said above, the restriction of $\mathcal{R}$ to $\Delta$ is the projective action of $A$. Since the line spanned by $\lambda^0$ is the eigenline of the (simple) largest eigenvalue of $A$, there exists $\alpha>1$ such that $(\mathrm{D}_{\lambda}\mathcal{R}_{\lambda})_{T_0}$ is $\alpha$-expanding;

\item The action of $(\mathrm{D}_{\mu}\mathcal{R}_{\mu})_{T_0}$ is that of $A$ restricted to the subspace defined by the equation $ \sum_i{\mu_i\lambda^i_0} = 0$. This space is stabilised by the action of $A$ and consequently the action of $(\mathrm{D}_{\mu}\mathcal{R}_{\mu})_{T_0}$ is diagonalisable with $g-1$ eigenvalues larger than $1$ and $g$ smaller than one. 

\end{enumerate}

\end{proof}

\noindent This proposition in particular implies that $T_0$ is a hyperbolic fixed point of $\mathcal{R}$ and that the unstable space at $T_0$ has dimension exactly $(d-1)+ (g-1)$.

\subsection{On the standing assumption.}
\label{assumption}

\noindent We wanted to point out that the assumption that $A$ be hyperbolic is not very restrictive. For any $d$ and combinatorics giving rise to a surface with only one marked point, there are infinitely many periodic $T_0$ and most of them have an intersection matrix which is hyperbolic. However, we would like to point out that it is not the case for all of them: Bressaud-Bufetov-Hubert have constructed infinitely many periodic IETs violating this condition, see \cite{BressaudBufetovHubert}.

\section{Estimates}
\label{estimates}

\noindent In this section we prove estimates on the distortion, the second derivative and third derivatives of iterated renormalisations. These will be crucial for the analysis of the renormalisation operator.

\subsection{Distortion bounds}

We prove a standard distortion lemma and apply it to show that the "profile" coordinate remains uniformly bounded under iteration of renormalisation. This fact will be the starting point of the correction operation carried out in Section \ref{corrections}.

\begin{lemma}
\label{bound1}
Let $T$ be a GIET. Let $J \subset [0,1]$ be an interval such that $J, T(J), T^2(J), \cdots, T^n(J)$ are pairwise disjoint and do not contain any singularities of $T$. Then for all $x,y \in J$ we have 

$$ \frac{\mathrm{D}(T^n)(x)}{\mathrm{D}(T^n)(y)} \leq \exp(\int_0^1{|\eta_T|\mathrm{dLeb}}). $$

\end{lemma}

\begin{proof}

The proof is classical. We have that $$ \log \mathrm{D}T^n(x) = \sum_{i=0}^{n-1}{\log \mathrm{D}T (T^i(x) )} $$ and therefore 

$$ | \log \mathrm{D}T^n(x) - \log \mathrm{D}T^n(y) | \leq  \sum_{i=0}^{n-1}{|\log \mathrm{D}T (T^i(x) ) -  \log \mathrm{D}T (T^i(x) ) |}  \leq  \sum_{i=0}^{n-1}{|\int_{T^i(y)}^{T^i(x)}{\eta_{T}}|}.$$  Since the intervals $[T^i(y), T^i(x)]$ are pairwise disjoints we get 

$$ | \log \mathrm{D}T^n(x) - \log \mathrm{D}T^n(y) | \leq \int_0^1{|\eta_T|\mathrm{dLeb}} $$ and exponentiating gives the expected result.

\end{proof}

\noindent In the sequel we use the following notation for $f : [0,1] \longrightarrow \mathbb{R}$ of class $\mathcal{C}^r$, 

$$ ||f||_{\mathcal{C}^r} = \max_{0\leq i \leq d}{||f^{(i)}||} $$ where $f^{(i)}$ is the $i$-th derivative of $f$. We extend this norm to $(\mathcal{C}^r([0,1], \mathbb{R}))^d$ simply by taking the sum of the norms on each coordinate. From the lemma above, we derive the following

\begin{proposition}
\label{bound2}
Recall that $\mathcal{X}^2$ was defined to be the set of $\mathcal{C}^2$-GIETs with a given combinatorial type. Let $K$ be a pre-compact set of $\mathcal{X}^2$ with respect to the $\mathcal{C}^2$-topology. There exists a constant  $M(K)>0$ such that for any $T$ GIET renormalisable $n$ times belonging to $K$ we have

$$ || \pi_{\mathcal{P}}\big( \mathcal{R}^n(T) \big ) - (\mathrm{Id})^d ||_{\mathcal{C}^1} \leq  M || \pi_{\mathcal{P}}(T) ||_{\mathcal{C}^2}  $$  where $(\mathrm{Id})^d  = (\mathrm{Id}, \cdots, \mathrm{Id}) \in \mathcal{P} = \big( \mathrm{Diff}^2_+([0,1]) \big)^d$.

\end{proposition}

\begin{proof}

Consider $\phi$ a coordinate of $\pi_{\mathcal{P}}\big( \mathcal{R}^n(T) \big )$. It is obtained by taking finitely many restrictions of $T$ to $k$ pairwise disjoint intervals $I_1, \cdots, I_k$, composing them and rescaling them. We can therefore apply Lemma \ref{bound1} to such a composition to find that for all $x,y \in [0,1]$, 

$$ \frac{D\phi(x)}{D\phi(y)} \leq \exp(\int_0^1{|\eta_T|\mathrm{dLeb}}).$$

\noindent Since $\phi$ is a diffeomorphism of $[0,1]$ there exists $z \in [0,1]$ such that $D\phi(z) = 1$. $T$ belongs to a precompact set with respect to the $\mathcal{C}^2$-topology so in particular $DT$ and $D(T^{-1})$ are bounded by a uniform constant. Because $\eta_T= \frac{D^2T}{DT}$ and the fact that the exponential is Lispchitz on compact sets of $\mathbb{R}$ we get the existence of a constant $L>0$ such that for all $x$

$$ \frac{D\phi(x)}{D\phi(y)} \leq L ||D^2T||.$$ Comparing an arbitrary point $x$ to $z$ gives the expected result.

\end{proof}

\subsection{$\mathcal{C}^2$-bounds}

In this paragraph we prove an estimate which give some uniform bounds on the second derivative of iterated renormalisation of elements in $\mathcal{X}$ close to $T_0$. The proof builds upon Lemma \ref{bound1}. To the best knowledge of the author, this estimate is  new.

\begin{lemma}
\label{secondderivative}
Let $\varphi_1, \cdots, \varphi_n \in \mathcal{C}^2(\mathbb{R}, \mathbb{R})$. For all $k \leq n$ define $f_k = \varphi_k \circ \varphi_{k-1} \circ \cdots \circ \varphi_1$ and set $f_0 = \mathrm{Id}$. Then we have for all $n \geq 2$ the formula 

$$ f_n'' =  (f'_{n-1})^2 \cdot (\varphi''_{n} \circ f_{n-1}) +  \sum_{k=2}^n{ (f'_{n-k})^2 \cdot (\varphi''_{n-k+1} \circ f_{n-k})  \cdot (\varphi_n \circ \cdots \circ \varphi_{n-k+2})' \circ f_{n-k+1}   } $$

\end{lemma}

\begin{proof}

We proceed by induction on $n$. We check that the statement holds true for $n = 2$: 

$$ f_2'' = (\varphi_2 \circ \varphi_1)'' = (\varphi'_1 \cdot \varphi_2'\circ \varphi_1)' = (\varphi'_1)^2 \cdot \varphi''_2 \circ \varphi_1 + \varphi''_1 \cdot \varphi_2' \circ \varphi_1.$$ 

\noindent Assume the statement holds true for $n \geq 2$. We have 

$$ f''_{n+1} = (\varphi_{n+1} \circ f_n)'' = \varphi''_{n+1} \circ f_n \cdot (f'_n)^2 + f''_n \cdot \varphi'_{n+1}\circ f_n.$$ Replacing $f''_n$ in the formula we get 

\begin{equation*}
\begin{split}
 & f''_{n+1} = (\varphi_{n+1} \circ f_n)'' =   \varphi''_{n+1} \circ f_n \cdot (f'_n)^2 +    (\varphi'_{n+1}\circ f_n) \cdot (f'_{n-1})^2 \cdot (\varphi''_{n} \circ f_{n-1}) \\
 &   + \sum_{k=2}^n{ (f'_{n-k})^2 \cdot (\varphi''_{n-k+1} \circ f_{n-k}) \cdot  (\varphi'_{n+1}\circ f_n) \cdot (\varphi_n \circ \cdots \circ \varphi_{n-k+2})' \circ f_{n-k+1}}. 
 \end{split}
\end{equation*} By the chain rule we have 

$$ (\varphi'_{n+1}\circ f_n) \cdot (\varphi_n \circ \cdots \circ \varphi_{n-k+2})' \circ f_{n-k+1}   =   ( \varphi_{n+1} \circ \varphi_n \circ \cdots \circ \varphi_{n-k+2})' \circ f_{n-k+1}$$ 

\noindent Injecting in the formula above for $f''_{n+1}$ gives the expected result.
\end{proof}

\noindent Consider a $\mathcal{C}^2$, increasing diffeomorphism $f : I \longrightarrow J$ where $I$ and $J$ are two connected intervals. We denote by $\mathrm{N}(f)$ the \textit{normalisation} or \textit{rescaling} of $f$, it is by definition the map $f$ pre-composed by the unique affine map sending $[0,1]$ onto $I$ and post-composed  by the unique affine map sending $J$ onto $[0,1]$. We have the following easy lemma:

\begin{lemma}
\label{normalised}
Let $f$ as above. Then we have 

$$ || \mathrm{N}(f)'' || \leq ||f'^{-1}|| \cdot ||f''|| \cdot |I| $$ 

\end{lemma}

\begin{proof}

\noindent Let $a = |I|$ and $b = |J|$. By definition we have 

$$ \mathrm{N}(f) := x \longmapsto \frac{1}{b} f(ax).$$ Thus 

$$ \mathrm{N}(f)''(x) = \frac{a^2}{b} f''(ax) = a \frac{a}{b} f''(ax).$$ There exists $x_0 \in I$ such that  $\frac{1}{f'(x_0)} = \frac{|I|}{|J|} = \frac{a}{b}$. Hence the result.

\end{proof}

\noindent Using Lemma \ref{secondderivative} and Lemma \ref{normalised}, we prove the following 

\begin{proposition}
\label{controlC2}
Let $\mathcal{V}$ be a pre-compact neighbourhood of $T_0$ in $\mathcal{X}^2$ with respect to the $\mathcal{C}^2$-topology. There exists a constant $M'$ such that the following holds. Let $T \in \mathcal{V}$ be a $\mathcal{C}^2$ GIET renormalisable $n$ times. We use the following notation $\pi_{\mathcal{P}}\big( \mathcal{R}^n(T) \big) = (\varphi_1^n, \cdots, \varphi_d^n) \in \big( \mathrm{Diff}_+^2([0,1]) \big)^d$. Then we have for all $i\leq d$ and for all $n \in \mathbb{N}$

$$ || (\varphi_i^n)'' || \leq M' ||(T^{-1})'|| \cdot  || T'' || $$  

\end{proposition}

\begin{proof}

The proof is an application of Lemma \ref{secondderivative} to the composition of restrictions of $T$ to the dynamical partition. Recall that $\varphi_i^n$ is the renormalised of $T^{l^i_n}$ restricted to an interval $I_n^i$ such that $I^i_n, T(I^i_n),  T^2(I^i_n), \cdots, T^{l^i_n-1}(I^j_n)$ are disjoint. We denote by $S_k$ the restriction of $T$ to $T^k(I^i_n)$. We have the following properties 

\begin{itemize}

\item $\varphi_i^n =\mathrm{N}(S_{l^j_n-1})  \circ \cdots \circ \mathrm{N}(S_1)  \circ \mathrm{N}(S_0) $ 

\item any partial product $\psi_k = \mathrm{N}(S_{l^j_n-1})  \circ \cdots \circ \mathrm{N}(S_k)$ is such that $||\log (\psi_k)'|| \leq K ||T''||$ ($\psi_k$ is a diffeomorphism of $[0,1]$ and therefore there exists $x_0 \in [0,1]$ such that $\log \psi'_k (x_0) = 0$ and the claim follows from Lemma \ref{bound1});

\item same holds for partial products $\phi_k = \mathrm{N}(S_{k})  \circ \cdots \circ \mathrm{N}(S_0)$;

\item for any $k$, $||\mathrm{N}(S_k)'|| \leq ||(T^{-1})'|| \cdot ||T''|| \cdot |T^k(I^j_n)| $.

\end{itemize}

\noindent The result is a consequence of Lemma \ref{secondderivative} applied to $\mathrm{N}(S_{l^i_n-1})  \circ \cdots \circ \mathrm{N}(S_1)  \circ \mathrm{N}(S_0)$. Indeed

$$  || \varphi_i^n)'' || \leq  \sum_{k=1}^{l^i_n}{  ||\phi_{n-k}'||^2 \cdot || \mathrm{N}''(S_{n-k+1})|| \cdot || \psi'_{n-k+2} ||  } $$ and replacing in the inequality 

$$ || \varphi_i^n)'' || \leq e^{3K||T''||} \cdot ||(T^{-1})'|| \cdot ||T''|| \sum_{k=0}^{l^i_n-1}{|T^k(I^i_n)|}.$$  The $T^k(I^j_n)$s are all disjoint, the $\exp$ is bounded on bounded sets and $||T''||$ is bounded (because $T$ belongs to a $\mathcal{C}^2$-precompact subset of $\mathcal{X}$) thus we get the result.

\end{proof}

\subsection{Bounds on  $D\eta$.}

We prove in this paragraph bounds on the function $D(\eta_T)$ along renormalisation when $r \geq 3$. The proofs follow the same line of thought as the previous section.

\begin{lemma}
\label{formulae1}
Let $\varphi_1, \cdots, \varphi_n \in \mathrm{Diff}_+^3([0,1])$. Set $\psi_k = \varphi_k \circ \varphi_{k-1} \circ \cdots \circ \varphi_1$ and $\psi_0 = \mathrm{Id}$. Then 

\begin{enumerate}

\item $ \log D(\varphi_n \circ \varphi_{n-1} \circ \cdots \circ \varphi_1) = \sum_{k=1}^n{\log D(\varphi_k) \circ \psi_{k-1}}$;

\item $\eta(\varphi_n \circ \varphi_{n-1} \circ \cdots \circ \varphi_1) = D(\log D(\varphi_n \circ \varphi_{n-1} \circ \cdots \circ \varphi_1)) = \sum_{k=1}^n{D \log D(\varphi_k) \circ \psi_{k-1} \cdot D\psi_{k-1}}$;

\item $D\eta(\varphi_n \circ \varphi_{n-1} \circ \cdots \circ \varphi_1)  =  \sum_{k=1}^n{D^2\log D(\varphi_k) \circ \psi_{k-1} \cdot (D\psi_{k-1} )^2 + D \log D(\varphi_k) \circ \psi_{k-1}  \cdot D^2\psi_{k-1}  }$;

\end{enumerate}
\end{lemma}

\noindent These formulae directly derive from the definition of the non-linearity $\eta(f) = D \log Df$ and their proofs are left to the reader. Let $f$ be a $\mathcal{C}^2$, increasing diffeomorphism $I \longrightarrow J$ where $I$ and $J$ are two connected intervals.

\begin{lemma}
Recall that $\mathrm{N}(f)$ is the rescaling of $f$. Then we have 

$$ |D\big(\eta(\mathrm{N}(f)) \big)| \leq (||f'''|| \cdot ||f'|| + ||f''||^2) \cdot  ||(f^{-1})'||^4 \cdot  |I|^2 $$

\end{lemma}

\begin{proof}
We have that 
$$ D\big(\eta(\mathrm{N}(f)) \big) = D (\frac{N(f)''}{N(f)'}) = \frac{N(f)''' N(f)' - (N(f)'')^2}{(N(f)')^2}$$. By the exact same reasoning as in the proof of Lemma \ref{normalised} we get that $||N(f)'''|| \leq ||f^{-1}|| \cdot  ||f'''|| \cdot  |I|^2$. We already had $||N(f)''|| \leq ||f^{-1}|| \cdot  ||f''|| \cdot |I|$ and because $N(f)^{-1} = N(f^{-1})$ we get the expected result.
\end{proof}

\noindent We are now ready to prove

\begin{proposition}
\label{controlC3}
Let $\mathcal{V}$ be a precompact neighbourhood of $T_0$ in the $\mathcal{C}^3$-topology. Let $T \in \mathcal{V}$ be a $\mathcal{C}^3$ GIET renormalisable $n$ times. We use the following notation $\pi_{\mathcal{P}}\big( \mathcal{R}^n(T) \big) = (\varphi_1^n, \cdots, \varphi_d^n) \in \big( \mathrm{Diff}_+^3([0,1]) \big)^d$. Then we have for all $i\leq d$ and for all $n \in \mathbb{N}$

$$ || D(\eta(\varphi_i^n))|| \leq K( \sup(||T''||,||T'''||))$$ where $K : \mathbb{R}^*_+ \longrightarrow \mathbb{R}^*_+$    is a continuous function which tends to $0$ in $0$.

\end{proposition}

\begin{proof}

Again we follow the lines of the proof of Proposition \ref{controlC2} but using formulae of Lemma \ref{formulae1}. Recall that $\varphi_i^n$ is the renormalised of $T^{l^i_n}$ restricted to an interval $I_n^i$ such that $I^i_n, T(I^i_n),  T^2(I^i_n), \cdots, T^{l^i_n-1}(I^i_n)$ are disjoint. We denote by $S_k$ the restriction of $T$ to $T^k(I^i_n)$. We have the following properties 

\begin{itemize}

\item $\varphi_i^n =\mathrm{N}(S_{l^i_n-1})  \circ \cdots \circ \mathrm{N}(S_1)  \circ \mathrm{N}(S_0) $ 

\item any partial product $\psi_k = \mathrm{N}(S_{l^i_n-1})  \circ \cdots \circ \mathrm{N}(S_k)$ is such that $||\log (\psi_k)'|| \leq K ||T''||$ ($\psi_k$ is a diffeomorphism of $[0,1]$ and therefore there exists $x_0 \in [0,1]$ such that $\log \psi'_k (x_0) = 0$ and the claim follows from Lemma \ref{bound1});

\item same holds for partial products $\phi_k = \mathrm{N}(S_{k})  \circ \cdots \circ \mathrm{N}(S_0)$;

\item for any $k$, $||\mathrm{N}(S_k)''|| \leq ||(T^{-1})'|| \cdot ||T''|| \cdot |T^k(I^i_n)| $;

\item $||(\phi_k)''|| \leq M' ||(T^{-1})'|| \cdot  || T'' ||$ by Lemma \ref{controlC2};

\item  $$ |D\big(\eta(\mathrm{N}(S_k)) \big)| \leq (||T'''|| \cdot ||T'|| + ||T''||^2) \cdot  ||(T^{-1})'||^4 \cdot   |T^k(I^i_n)|^2 $$

\end{itemize}

We can now apply the third formulae of Lemma \ref{formulae1} to the product $\varphi_i^n =\mathrm{N}(S_{l^i_n-1})  \circ \cdots \circ \mathrm{N}(S_1)  \circ \mathrm{N}(S_0) $ to get

$$ D\eta(\varphi_i^n) =   \sum_{k=0}^{l^i_n-1}{D^2\log D(N(S_k)) \circ \phi_{k-1} \cdot (D\phi_{k-1})^2 + D \log D(N(S_k)) \circ \phi_{k-1} \cdot D^2\phi_{k-1} }.$$ Recall that $D^2\log D(N(S_k))  = D\eta(N(S_k))$ and $\eta(N(S_k)) = \frac{N(S_k)''}{N(S_k)'}$. Putting all the inequalities above together we get 

\begin{equation*}
\begin{split}
 | D\eta(\varphi_i^n) | \leq &  \exp ( K ||T''||) \cdot \sum_k{(||T'''|| \cdot ||T'|| + ||T''||^2) \cdot  ||(T^{-1})'||^4 \cdot   |T^k(I^i_n)|^2 } \\ 
 & +  M' ||(T^{-1})'|| \cdot  || T'' ||\sum_k{ || (N(S_k)^{-1})'|| \cdot ||(T^{-1})'|| \cdot ||T''|| \cdot |T^k(I^i_n)|}.
\end{split}
\end{equation*}

\noindent Finally, because $||(N(S_k)^{-1})'||$ is uniformly controlled by $\frac{||T'||}{||(T^{-1})'||}$, that $\sum{|T^k(I^i_n)|}$ and  $\sum{|T^k(I^i_n)|^2}$ are smaller than $1$ and that $T$ belongs to a bounded $\mathcal{C}^3$ neighbourhood of $T_0$, we get the expected result.
\end{proof}

\section{Construction of a pre-stable space}
\label{corrections}

This section is the heart of the article. We construct what we call a "pre-stable" space which is a submanifold of $\mathcal{X}$ of codimension $d- 1 + g-1$, satisfying \textit{a priori bounds} for the geometry of the dynamical partitions. \textbf{We now make the standing assumption} $r = 3$.

\subsection{Notations and preliminaries} In the sequel, we place ourselves in a neighbourhood $\mathcal{W}$ of $T_0$ for the $\mathrm{C}^3$-topology. Up to restricting this neighbourhood further, we can identify it with an open neighbourhood of $0$ in the Banach space upon which $\mathcal{X} = \mathcal{A} \times \mathcal{P}$ is modelled. In these coordinates, we will use the notation $T_0 = (0_{\mathcal{A}},  0_{\mathcal{P}})$ where $0_{\mathcal{P}}$ represents the point $(\mathrm{Id}, \mathrm{Id}, \cdots,  \mathrm{Id}) \in \mathrm{Diff}^r_+([0,1])$ and $0_{\mathcal{A}}$ represents $T_0$ seen as an element of $\mathcal{A}$.  Note that 

$$ \eta_{T_0} \equiv 0 $$ therefore by restricting $\mathcal{W}$ further we can assume that 

$$\forall T \in \mathcal{W}, \ || \eta_T || \leq \epsilon $$ for any choice of a positive $\epsilon$ (this is possible since $r =3$).

\vspace{2mm} \paragraph*{\bf Some more notation} We then write a neighbourhood of $T_0$ in $\mathcal{A}$ as a product $\mathcal{U} \times \mathcal{S}$ where $\mathcal{U}$ is the subspace of unstable directions of $\mathcal{R}$ at $T_0$ and  $\mathcal{S}$ is the subspace of stable directions. Consequently, we identify a neighbourhood of $T_0$ in $\mathcal{X}$ to a product $\mathcal{S} \times \mathcal{U} \times \mathcal{P}$ where $\mathcal{P}$ abusively denotes (a neighbourhood of $0$ in) the Banach space upon which $(\mathrm{Diff}^r_+([0,1]))^d$ is modelled. 
\noindent In these coordinates, we write 

$$ \mathcal{R} = (\mathcal{R}_{\mathcal{A}}, \mathcal{R}_{\mathcal{P}}) = (\mathcal{R}_{\mathcal{S}}, \mathcal{R}_{\mathcal{U}}, \mathcal{R}_{\mathcal{P}}).$$

\noindent Finally, we denote by $\pi_{\mathcal{A}}, \pi_{\mathcal{S}}, \pi_{\mathcal{U}}$ and $\pi_{\mathcal{P}}$ the projection from $\mathcal{X}$ onto $\mathcal{A}, \mathcal{S}, \mathcal{U}$ and $\mathcal{P}$ respectively.

\subsection{Action of $\mathcal{R}$.} Recall from Section \ref{affine} that $T_0$ is a hyperbolic fixed point of $\mathcal{R}$ restricted to $\mathcal{A}$. We collect in this paragraph important properties of $\mathcal{R}$.

\begin{enumerate}

\item $\mathcal{R}(0) = 0$;

\item $\mathcal{R}(\mathcal{A}) = \mathcal{A}$;

\item $\mathcal{R}$ is continuous;

\item $\mathcal{R}_{\mathcal{A}}$ is of class $\mathcal{C}^1$;

\item $0_A$ is a hyperbolic fixed point of $\mathcal{R}$ restricted to $\mathcal{A}$;

\item $\mathrm{D}\mathcal{R}_{\mathcal{A}}$ is a bounded operator.
\end{enumerate}

\noindent A difficulty that we face is that $\mathcal{R}$ is not smooth, it is not derivable in the $\mathcal{P}$ direction. It is a simple consequence of the fact that the map $(\varphi,\psi) \mapsto \varphi \circ \psi$  

$$ \mathrm{Diff}_+^r([0,1]) \times \mathrm{Diff}_+^r([0,1]) \longrightarrow \mathrm{Diff}_+^r([0,1]) $$ is not differentiable. To be able to perform the construction to come, we nonetheless need some control on this map.

%
%

\vspace{3mm}

\paragraph{\bf An appropriate choice of a distance}
 Recall that $ \mathrm{Diff}_+^r([0,1]) $ is Banach manifold whose tangent space at any point identifies with the Banach space $\mathcal{C}^r_0([0,1], \mathbb{R})$ of $\mathcal{C}^r$ real-valued functions which vanish at $0$ and $1$. 

\noindent We endow  $ \mathrm{Diff}_+^2([0,1]) $ with the following distance 

$$ d_{\eta}(f,g) = \int_0^1{|\eta_f - \eta_g|}.$$ Because the tangent space at any point of $ \mathrm{Diff}_+^2([0,1]) $ identifies with $\mathcal{C}^2_0([0,1], \mathbb{R})$ we can also use the formula $\int_0^1{|\eta_f - \eta_g|}$ to define a distance on $\mathcal{C}^2_0([0,1], \mathbb{R})$. We refer to it as the $\eta$-distance. It is more refined than the $\mathcal{C}^1$-norm but less than the $\mathcal{C}^2$-norm.   We then endow $\mathcal{P} = (\mathrm{Diff}_+^3([0,1]))^d $ with the $\eta$-distance: precisely, if $\varphi = (\varphi_1, \cdots, \varphi_d) \in \mathcal{P}$ and $\psi = (\psi_1, \cdots, \psi_d) \in \mathcal{P}$ then 

$$ d_{\eta}(\varphi, \psi) = \sum_{i=1}^d{d_{\eta}(\varphi_i, \psi_i)}.$$

\begin{proposition}
\label{lipschitz}
For any $\delta >0$ there exists a neighbourhood (for the $\mathcal{C}^3$-norm) of $T_0$ such that the restriction of $\mathcal{R}_{\mathcal{P}}$ to the $\mathcal{P}$ coordinates is $(1+ \delta)$-Lipschitz, with respect to $d_{\eta}$, restricted to this neighbourhood. 
\end{proposition}

\noindent The key to the proof of this proposition are the following facts

\begin{lemma}
\label{nonlinearity}
For any two function $f,g \in \mathcal{C}^2(\mathbb{R}, \mathbb{R})$, we have 

\begin{enumerate}

\item for any real number $a$, $\eta(a \cdot f) = \eta(f)$;

\item for any real number $a$, $ \eta( f \circ m_a) =  a \cdot \eta(f) \circ m_a$ where $m_a := x \mapsto ax$;

\item $\eta(f \circ g) = g' \cdot \eta_f \circ g + \eta(g)$.

\end{enumerate}

\end{lemma}

\noindent We leave the proof of these elementary statements to the reader. We are now ready to give the proof of Proposition \ref{lipschitz}.

\begin{proof}
Fix $\epsilon>0$. Let $T_1$ and $T_2$ be two GIETs close to $T_0$ such that $\pi_{\mathcal{A}}(T_1) =\pi_{\mathcal{A}}(T_2)$ . Let $\pi_{\mathcal{P}}(T_1)= (\varphi_1^1, \cdots, \varphi_d^1)$ and $\pi_{\mathcal{P}}(T_2)= (\varphi_1^2, \cdots, \varphi_d^2)$. We want to show that 

$$ d_{\eta}(\mathcal{R}_{\mathcal{P}}(T_1),\mathcal{R}_{\mathcal{P}}(T_2)) \leq (1+\epsilon) d_{\eta}(\pi_{\mathcal{P}}(T_1), \pi_{\mathcal{P}}(T_2)) $$ provided $T_1$ and $T_2$ are in a sufficiently small $\mathcal{C}^2$-neighbourhood of $T_0$. We have the following facts

\begin{enumerate}

\item for all $i$, $||(\varphi_i^1)'- (\varphi_i^2)'||_0 \leq K_1 d_{\eta}(\varphi_i^1, \varphi_i^2)$ for a certain constant $K_1$;

\item for all $i$, $||\varphi_i^1- \varphi_i^2||_0 \leq K_2 d_{\eta}(\varphi_i^1, \varphi_i^2)$ for a certain constant $K_2$;

\item the symmetric difference of the dynamical partition associated with $T_1$ and $T_2$ is less than $K_0 \cdot \sup_{i}{||\varphi_i^1- \varphi_i^2||_0}$ where $K$ is a uniform constant depending on the combinatorics of the dynamical partition only.

\end{enumerate}

\noindent The first two facts derive from the facts that $\eta_{f} = \frac{f''}{f'}$ and that we are in a $\mathcal{C}^2$-neighbourhood of $T_0$. Let us give a proof of the first fact. First we show that there exists $x_0$ such that $(\varphi_i^1)'(x_0) = (\varphi_i^2)'(x_0)$. This derives from the fact that if it were never the case we would have $(\varphi_i^1)'(x) > (\varphi_i^2)'(x)$ (or the opposite inequality) for all $x$, contradicting that the range of both $\varphi_1$ and $\varphi_2$ is $[0,1]$. Now for all $x \in [0,1]$

$$ | \log D\varphi_1(x) - \log D\varphi_2(x) = | \int_{x_0}^{x}{\eta(\varphi_1) - \eta(\varphi_2)} | \leq  d_{\eta}(\varphi_i^1, \varphi_i^2) $$ and we get $(1)$ because the exponential map is Lipschitz on bounded sets. The second point is proved in a similar fashion. The third fact is a consequence of the first two facts together with the hypothesis $\pi_{\mathcal{A}}(T_1) =\pi_{\mathcal{A}}(T_2)$.

\noindent We now want to find an estimate of

$$d_{\eta}(\psi_1, \psi_2) =  \sum_i{ \int_0^1{|\eta_{\psi_i^1} - \eta_{\psi_i^2}|} }$$ where $\pi_{\mathcal{P}}(\mathcal{R}(T_1))= (\psi_1^1, \cdots, \psi_d^1)$  and $\pi_{\mathcal{P}}(\mathcal{R}(T_2))= (\psi_1^2, \cdots, \psi_d^2)$.    The strategy is to decompose this sum in order to rewrite it as a new sum of integral of difference of the form $|\eta_{\phi_i^1} - \eta_{\phi_i^2}|$ over the dynamical partition, neglecting the subset of $[0,1]$ for which the the dynamical partition of $T_1$ differs from that of $T_2$. First, let us point out that because of Lemma \ref{nonlinearity}, all the quantities we are dealing with are invariant by  rescaling of the $\varphi_i^{\epsilon}$ at the source and/or at the target by affine maps ($\eta$ scales by a factor $a$ when the source is scaled by $a$ but the Lebesgue measure scales by $\frac{1}{a}$ which makes $\int{\eta}$ globally invariant). We can therefore think of the $\varphi_i^{\epsilon=1,2}$s as the (non-rescaled) restrictions of $T_{\epsilon=1,2}$ to its branches.

\vspace{2mm}
\noindent 
If $I_1^1, \cdots  I_d^1$ and $I_1^2, \cdots  I_d^2$  are the base intervals of the respective partitions associated with $T_1$ and $T_2$, let $J_i = I_i^1 \cap I_i^2$ for all $i$. By the facts stated above we have that the iterated images(up to times defining $\mathcal{R}$) of the $J_i$s cover all of $[0,1]$ up to a set of measure at most $K_0 \cdot d_{\eta}(\varphi^1, \varphi_2)$. Therefore we have that 
$d_{\eta}(\psi^1, \psi_2) \leq \sum_i{ \int_{\cup{J_i}}{|\eta_{\psi_i^1} - \eta_{\psi_i^2}|} } + \int_{Q}{|\eta_{T_1}|} + \int_{Q}{|\eta_{T_2}|}$ where $Q = [0,1] \setminus \cup{J_i} $. Each $\psi_i^1$ (and respectively $\psi_i^2$) is a composition of restrictions of $T_1$ (respectively $T_2$) to elements of dynamical partitions. Recall that by Lemma \ref{nonlinearity} we have for any two functions $f,g$ 

$$ \eta(f \circ g) = g' \cdot \eta_f \circ g + \eta(g).$$ Assume for the sake of simplicity that $\psi_i^1$ and $\psi_i^2$ are obtained by composition only two restrictions of $T_1$ and $T_2$. We would then have 

$$\int_{I_j}{|\eta_{\psi_i^1} - \eta_{\psi_i^2}|} = \int_{I_j}{|\eta_{(T_1)^2} - \eta_{(T_2)^2}|}.$$ Injecting using the composition formula gives 

$$\int_{I_j}{|\eta_{\psi_i^1} - \eta_{\psi_i^2}|} = \int_{I_j}{|\eta_{T_1} + DT_1\cdot \eta_{T_1} \circ T_1 - (\eta_{T_2} + DT_2 \cdot \eta_{T_2} \circ T_2)|}$$ and we get

$$ \int_{I_j}{|\eta_{\psi_i^1} - \eta_{\psi_i^2}|} \leq \int_{I_j}{|\eta_{T_1} - \eta_{T_2}|} + \int_{I_j}{DT_1 |\eta_{T_1}\circ T_1 - \eta_{T_2}\circ T_1|} + \int_{I_j}{DT_1 |\eta_{T_2}\circ T_1 - \eta_{T_2}\circ T_2|} + \int_{I_j}{|DT_1 - DT_2| \cdot |\eta_{T_2}\circ T_2|} .$$

\noindent To control each term of this sum we use the following facts

\begin{itemize}
\item a simple change of variable gives $\int_{I_j}{DT_1 |\eta_{T_1}\circ T_1 - \eta_{T_2}\circ T_1|}  = \int_{T_1(I_j)}{ |\eta_{T_1} - \eta_{T_2}|}$;

\item $ \int_{I_j}{DT_1 |\eta_{T_2}\circ T_1 - \eta_{T_2}\circ T_2|}  \leq ||DT_1|| \int_{I_j}{||D\eta_{T_2}|| \cdot |T_1 - T_2| } \leq |I_j| K_2 ||DT_1|| ||D\eta_{T_2}|| d_{\eta}(\varphi_j^1,\varphi_j^2) $ since $T_1$ restricted to $I_j$ is equal to $\varphi_j^1$ up to rescaling;

\item Finally $\int_{I_j}{|DT_1 - DT_2| \cdot |\eta_{T_1}\circ T_1|} \leq ||\eta_{T_1}|| \cdot || DT_1 - DT_2|| \leq ||\eta_{T_1}||  K_1 d_{\eta}(\varphi_j^1,\varphi_j^2) $.

\end{itemize}

\vspace{2mm} Putting everything together and by taking a sufficiently small $\mathcal{C}^3$-neighbourhood we get

$$\int_{I_j}{|\eta_{\psi_i^1} - \eta_{\psi_i^2}|} \leq  \int_{I_j \cup T_1(I_j)}{|\eta_{T_1} - \eta_{T_2}|} + \frac{\epsilon}{d} d_{\eta}(\varphi_j^1,\varphi_j^2).$$
\noindent This reasoning directly carries over to the case where $\psi_i^1$ and $\psi_i^2$ are obtained by a fixed but arbitrarily larger number of iterations of $T_1$ and $T_2$. We thus obtain that

$$ \sum_i{ \int_0^1{|\eta_{\psi_i^1} - \eta_{\psi_i^2}|}}  \leq (1+ \epsilon) \sum_i{ \int_0^1{|\eta_{\varphi_i^1} - \eta_{\varphi_i^2}|}} $$ which is the expected result.

\end{proof}

\subsection{Invariant cones} We now construct a continuous family of cones in a neighbourhood of $T_0$ which are invariant for the action of $\mathcal{R}$ on $\mathcal{X}$. Recall that we are using the distance  $d_{\eta}$ on the coordinate $\mathcal{P}$. This distance induces a distance of $\mathcal{X}^2 = \mathcal{A} \times \mathcal{P}$ (the space of twice continuously differentiable GIETs). 

\begin{itemize}
\item In the sequel, we restrict our attention to a neighbourhood of $T_0$ for the $\mathcal{C}^3$-topology.

\item The Banach space structure of $\mathcal{X}$ is induced by an identification of $\mathcal{X}$ with an open subset of an affine space modelled on $\mathbb{R}^{2d-2} \times (\mathcal{C}^2_0([0,1]))^d$. Consider any norm $|| \cdot ||$ on $\mathcal{A} \simeq \mathbb{R}^{2d-2}$ which derives from a scalar product and makes the stable and unstable spaces in $\mathcal{A}$ of $\mathcal{R}$ at $T_0$ orthogonal and make the product with $d$ times the $\eta$-distance to get a distance on $\mathbb{R}^{2d-2} \times (\mathcal{C}^2_0([0,1]))^d$.
  
\item The neighbourhood $\mathcal{W}$ of $T_0$ in $\mathcal{X}$ identifies canonically with a neighbourhood of $0$ in $\mathbb{R}^{2d-2} \times (\mathcal{C}^3_0([0,1]))^d$. In this section we make this identification; the $\eta$-distance is thus the distance induced by the $\eta$-distance on $\mathbb{R}^{2d-2} \times (\mathcal{C}^2_0([0,1]))^d$.

\item We will  use coordinates $(s,u,h) \in \mathcal{S} \times \mathcal{U} \times (\mathcal{C}^2_0([0,1]))^d$ in this identification. In particular $h = (h_1, \cdots, h_d)$ corresponds to diffeomorphisms $(\mathrm{Id} + h_1, \cdots, \mathrm{Id} + h_n) \in \mathcal{P}$.
\end{itemize}

\noindent  For any $x \in \mathcal{W}$ and any $\delta>0$ we define the following cone 

$$ C_x^{\delta} := \{ x+ u + (s+h) \ | \ u\in \mathcal{U}, \ s \in \mathcal{S}, \ h \in (\mathcal{C}^2_0([0,1]))^d\ \text{and} \  ||s|| \leq \delta ||u||, d_{\eta}(\pi_{\mathcal{P}}(x), \pi_{\mathcal{P}}(x) +h) \leq \delta  ||u||   \}. $$

\begin{lemma}[Invariant cones]
\label{lemma1}
There exists $\lambda_1 >1$, $\delta>0$, $\epsilon_1>0$ and $\alpha_1 >0$ such that, up to restricting $\mathcal{W}$ further we have that $\forall x \in \mathcal{W}$

\begin{enumerate}

\item $\mathcal{R}(C_x^{\delta}\cap B_x(\epsilon_1)) \subset \mathrm{Int}(C_{\mathcal{R}(x)}^{\delta}) $;

\item $\mathcal{R}$ restricted to $C_x^{\delta}\cap B_x(\epsilon_1)$ is $\lambda_1$-expanding.

\end{enumerate}

(Balls considered here are balls with respect to the $\eta$-distance).

\end{lemma}

\begin{proof}
Note that both properties are open in $x$, so we only have to check that these are true in $0$. Recall that $\mathcal{R}_{\mathcal{A}}$ is of class $\mathcal{C}^1$. We have the following facts
\begin{enumerate}

\item $(\mathrm{D}_{\mathcal{U}}\mathcal{R}_{\mathcal{U}})_0$ is $\lambda$-expanding for a certain $\lambda > 1$;

\item $(\mathrm{D}_{\mathcal{S}}\mathcal{R}_{\mathcal{U}})_0 =0$;

\item $(\mathrm{D}_{\mathcal{U}}\mathcal{R}_{\mathcal{S}})_0 = 0$;

\item up to rescaling coordinates we can ensure $|| (\mathrm{D}_{\mathcal{P}}\mathcal{R}_{\mathcal{U}})_0 || \leq  1$ and   $|| (\mathrm{D}_{\mathcal{P}}\mathcal{R}_{\mathcal{S}})_0 || \leq  1$;

\item $(\mathrm{D}_{\mathcal{S}}\mathcal{R}_{\mathcal{S}})_0$ is contracting.
\end{enumerate}

\noindent Consider $u \in \mathcal{U}$ and $(s,h) \in \mathcal{S} \times \mathcal{P}$ such that $||s|| \leq \delta ||u||$ and $d_{\eta}(\pi_{\mathcal{P}}(x), \pi_{\mathcal{P}}(x) +h) \leq \delta ||u||$. Recall that $\mathcal{R}_{\mathcal{U}}$ is differentiable, with respect to the $\mathcal{C}^1$-norm in the coordinate $h$ and that the $\mathcal{C}^1$-norm is controlled by the $\eta$-distance \textit{i.e.} there exists a uniform constant $K$ such that for all $h_1$ and $h_2$ is a bounded neighbourhood of $0$, we have $||h_1 - h_2||_{\mathcal{C}^1} \leq K d_{\eta}(\mathrm{Id} + h_1, \mathrm{Id} + h_2)$. We have

$$ \mathcal{R}_{\mathcal{U}}(u,s,h) = (\mathrm{D}_{\mathcal{U}}\mathcal{R}_{\mathcal{U}})_0(u) + (\mathrm{D}_{\mathcal{P}}\mathcal{R}_{\mathcal{U}})_0(h) + o(||u||) $$ and by restricting to a small enough ball we get $ ||\mathcal{R}_{\mathcal{U}}(u,s,h) || \geq (\lambda - \delta - \epsilon)||u||$ for any arbitrarily fixed $\epsilon$. Then we have 

$$ \mathcal{R}_{\mathcal{S}}(u,s,h) = (\mathrm{D}_{\mathcal{S}}\mathcal{R}_{\mathcal{S}})_0(s) + (\mathrm{D}_{\mathcal{P}}\mathcal{R}_{\mathcal{S}})_0(h) + o(||u||) $$ from which we get 

$$ || \mathcal{R}_{\mathcal{S}}(u,s,h) || \leq (\delta+\epsilon) ||u|| $$ Finally 

$$  d_{\eta}\big( \mathcal{R}_{\mathcal{P}}(u,s,h) - \mathcal{R}_{\mathcal{P}}(u,s,0) \big)  \leq (1+\epsilon) d_{\eta}\big( \mathrm{Id}, \mathrm{Id} + h \big) $$ because the restriction of $\mathcal{R}_{\mathcal{P}}$ to the variable $\mathcal{P}$ can be taken made $(1+\epsilon)$-Lipschitz by restricting $\mathcal{W}$ further (this is given by  Proposition \ref{lipschitz}). But we have that $ \mathcal{R}_{\mathcal{P}}(u,s,0) = 0$ which gives

$$ d_{\eta}\big( \mathcal{R}_{\mathcal{P}}(u,s,h), \mathrm{Id} \big)  \leq (1+\epsilon) \cdot d_{\eta}\big( \mathrm{Id}, \mathrm{Id} + h \big)  $$   Taking $\epsilon$  and $\delta$ small enough (such that $1+\epsilon < \lambda - \delta -\epsilon$), we get the expected result.

\end{proof}


%
%
%
%
%
%
%
%

In what follows we will get rid of the dependency in $\delta$ in the notation and use the notation $C^{\delta}_x = C_x$. We now turn to prove a lemma that is going to be the technical cornerstone we will rely upon in the course of the construction the "pre-stable" space.

\begin{lemma}
\label{lemma2}
There exists $\lambda_2 > 1$ such that for all $x = (s,u,p) \in \mathcal{W}$ such that $\forall k \leq n$, $\mathcal{R}^n(x) \in \mathcal{W}$ the following holds true.
\noindent Set $\mathcal{R}^n(x) = (s_n,u_n,p_n)$. Pick $u'$ such that $||u_n -u'|| \leq \epsilon_1$. Then there exists $v_n$ such that 

\begin{itemize}
\item $||v_n|| \leq \lambda_2^{-n}||u_n -u'||$;

\item $\pi_{\mathcal{U}}(\mathcal{R}^n(s,u+v_n,p)) = u'$

\item for all $k\leq n$, $d_{\eta} \big( \mathcal{R}^k(s,u+v_n,p) -\mathcal{R}^k(s,u,p) \big) \leq  \lambda_2^{k-n}||u_n -u'||$.

\item $v_n$ depends continuously on $s,u,p$ and $u'$.
\end{itemize}

\end{lemma}

\noindent Just before entering the proof of this lemma, we comment on the qualitative meaning of it. This lemma essentially tells us that initial perturbations in the $\mathcal{U}$-direction propagate at an exponential rate in the $\mathcal{U}$-direction and allow for cheaper and cheaper corrections as we renormalise further and further.

\begin{proof}

Let $B_0 \subset \mathcal{U}$ the ball of radius $\epsilon_1$ in $\mathcal{U}$ centred at $T_0 = 0$ and let $D_0 = x+B_0$. The image of $D_0$ under the action of $\mathcal{R}$ is an embedded ball of dimension  $\mathrm{dim}(\mathcal{U}) = d-1 + g-1$ enjoying the following properties 

\begin{itemize}

\item it projects injectively onto a neighbourhood of $0$ in $\mathcal{U}$ (with the coordinate re-centred to $\mathcal{R}(x)$);

\item at any point of $y \in \mathcal{R}(D_0)$, there is a neighbourhood of $y$ in $\mathcal{R}(D_0)$ which is contained in $C_y$.

\end{itemize}

\noindent These two properties are a consequence of Lemma \ref{lemma1}. We now consider the set of points of $\mathcal{R}(D_0)$ which project onto the ball of radius $\epsilon_1$ in $\mathcal{U}$(ball centred at $\mathcal{R}(x)$); we call this set $D_1$. Using the same construction we can construct $D_2$ which is the set of points in $\mathcal{R}(D_1)$ which project onto the ball of radius $\epsilon_1$ in $\mathcal{U}$ (ball centred at $\mathcal{R}^2(x)$). Again, by applying Lemma  \ref{lemma1}  we get that this set is a ball which has a neighbourhood at $y$ that is contained in $C_y$ for all $y$. We thus construct the sequence $(D_i)_{i\leq n}$ satisfying the following

\begin{itemize}

\item for all $i \leq n$, $D_i$ is a embedded ball of dimension $\mathrm{dim}(\mathcal{U})$ containing $\mathcal{R}^n(x)$;

\item for all $i \leq n$, $D_i \subset \mathcal{C}_{\mathcal{R}^i(x)}$

\item for all $i \leq n$, $D_{i+1} \subset \mathcal{R}(D_i)$;

\item the restriction of  $\mathcal{R}$ to each $\mathcal{R}^{-1}(D_{i})$ is $\lambda_2$-expanding for a certain $1 < \lambda_2 \leq \lambda_1$. 

\item for all $i \leq n$, $D_i$ projects bijectively on the ball of radius $\epsilon_1$ centred at $\mathcal{R}^i(x)$ in $\mathcal{U}$.

\end{itemize} 

\noindent Since $||u' - u_n|| < \epsilon_1$, there exists $x'_n \in D_n$ such that $\pi_U(x'_n) = u'$. By considering the iterated pre-images of $x_n$ by $\mathcal{R}$ we find $v_n$ such that  $\pi_{\mathcal{U}}(\mathcal{R}^n(s,u+v_n,p)) = u'$. Since $\mathcal{R}$ is $\lambda_2$-expanding restricted to $D_i$ for all $i$, we get the conclusions of the Lemma.

\vspace{2mm} Continuity of $v_n$ comes from that of $\mathcal{R}$.

\end{proof}

\subsection{Construction of the pre-stable space}

 \noindent In this paragraph we prove the following theorem.

\begin{thm}
\label{prestable}
There exists a continuous function $\phi : \mathcal{W}' \subset\mathcal{S} \times \mathcal{P} \longrightarrow \mathcal{U}$ and a positive constant $K_1$ such that 

$$ \forall n \in \mathbb{N}, \ \forall (s,h) \in S \times \mathcal{P}, \  || \mathcal{R}^n(s, \phi(s,h),h) ||_{\mathcal{C}^1} \leq  K_1 $$  where $\mathcal{W}'$ is a neighbourhood of $0$ in $\mathcal{S} \times \mathcal{P}$ for the topology induced by the $\mathcal{C}^2$-norm. 
\end{thm}

\noindent A couple of comments before entering the proof of Theorem \ref{prestable}

\begin{enumerate}

\item This "pre-stable" space is a submanifold for which corresponding GIETs satisfy \textit{a priori bounds} or in other words a "Denjoy-Koksma" inequality for the logarithm of the derivative. This means that derivatives at the special times corresponding to the induction are uniformly bounded above and below away from zero.

\item  The codimension of this pre-stable space is exactly that of the stable space for the renormalisation restricted to AIETs. 

\item We actually prove that the renormalisation in this pre-stable space remain bounded with respect to the $\mathcal{C}^3$ distance, which is stronger than the $\mathcal{C}^1$. 

\end{enumerate}

\begin{proof}

We first make the following general remark. If we consider $\mathcal{W}$ a neighbourhood of $T_0$ in $\mathcal{X}$ for the $\mathcal{C}^3$-norm, we know by Proposition \ref{bound2}, Proposition \ref{controlC2} and Proposition \ref{controlC3} that for any $T \in \mathcal{W}$, $\mathcal{R}^n_{\mathcal{P}}(T)$ remains in a small neighbourhood of $(\mathrm{Id}, \cdots, \mathrm{Id})^d \in \mathcal{P} = (\mathrm{Diff}^3_+[0,1])^d$ in the $\mathcal{C}^3$-norm. This is a very important point as the construction of invariant cones only works for a neighbourhood of $T_0$ in the $\mathcal{C}^3$-norm, even though we are working in practice with the $\mathcal{C}^1$-norm. Thus, to show that the sequence  $\mathcal{R}^n(T)$ stays close to $T_0$ we only need to check that the projection on $\mathcal{A}$ stays close to $T_0$.

\vspace{3mm}

\noindent We consider $\epsilon \leq \epsilon_1$ from Lemma \ref{lemma2} and $(s,h) \in \mathcal{S} \times \mathcal{P}$ such that $|| h ||_{\mathcal{C}^2} \leq \frac{\epsilon}{2M}$ and $||s|| \leq \frac{\epsilon}{2}$, where $M$ is the constant of Proposition \ref{bound2}. We warn the reader that we will restrict $\epsilon$ further in the course of the proof.

\noindent We build the function $\phi$ by an inductive process which consist in adding, for all $n \geq 0$, small perturbations in order to compensate for the error in the unstable direction that is brought by the non-vanishing of the "profile" coordinate. The ultimate goal is to show that the sum of all these corrections converges.

\noindent  Set $V_0 = 0$. We write $\mathcal{R}(s, 0, h) = (s_1, u_1, h_1)$ with $$ || s_1 || \leq \frac{\epsilon}{2}, \ || u_1 || \leq \frac{\epsilon}{2} \ \text{and} \ || h_1 ||_{\mathcal{C}^1} \leq \frac{\epsilon}{2}.$$ For the remainder of the proof, the norm we use in the $\mathcal{P}$-coordinate is the $\mathcal{C}^1$-norm. The fact that $|| h_1 ||_{\mathcal{C}^1} \leq \frac{\epsilon}{2}$ is a consequence of Proposition \ref{bound2}.

\vspace{3mm}

\noindent There exists constants  $K_2,K_3>0$ such that in a $\mathcal{C}^1$-neighbourhood of  $(0,0,0)$ we have

$$|| \pi_S(\mathcal{R}(s,u,h)) || \leq \lambda_2^{-1} ||s|| + K_2 ||u|| + K_3||h||_{\mathcal{C}^1}.$$ Therefore if we restrict $h$ further so that its $\mathcal{C}^2$ norm is less than $ \frac{\epsilon(1-\lambda_2^{-1})}{2MK_3}$ and by applying Lemma \ref{lemma2} we get the existence of $v_1 \in \mathcal{U}$ such that 

\begin{itemize}

\item $|| v_1 || \leq  \lambda_2^{-1} \frac{\epsilon}{2}$;

\item $ \mathcal{R}(s, v_1  , h) = (s_1', 0, h_1' )$

\item $||s_1'|| \leq \frac{\epsilon}{2}$

\end{itemize}

\noindent We are now in a good position to iterate the process.

\vspace{3mm}

\noindent Set $V_1 = v_1$. We define inductively $ V_{n+1} = V_n + v_{n+1} $ by making the choice of $v_{n+1}$ explained below.  We want the three following properties

\begin{enumerate}

\item for all $k \leq n$, $ || \pi_S(\mathcal{R}^k(s, V_n,h)) || \leq  \frac{\epsilon}{2} $;

\item for all $k \leq n$, $ || \pi_U(\mathcal{R}^k(s, V_n, h)) || \leq  \frac{\epsilon}{2} \sum_{i= 0}^{n- k}{\lambda_2^{-i}} $;

\item $\pi_U(\mathcal{R}^n(s, V_n, h))= 0$.

\end{enumerate}

\noindent We write

$$ \mathcal{R}^n(s,V_n,h) = (s'_n,0,h'_n).$$

\noindent Note that since $||h||_{2} \leq \frac{\epsilon}{2M}$, $||h_n||_1 \leq \frac{\epsilon}{2}$ by Proposition \ref{bound2}.  Also $s_n \leq \frac{\epsilon}{2}$ by the same reasoning as in the first step described above. We therefore get that $\mathcal{R}^{n+1}(s,V_n,h) = \mathcal{R}(s'_n,0,h'_n) = (s_{n+1}, u_{n+1}, h_{n+1})$ with $||u_{n+1}|| \leq K_4 \epsilon$ for a certain constant $K_4$. \footnote{This is the key argument. Because of the distortion bounds and Proposition \ref{bound2} , $||h_n||$ is uniformly small. In turn, because $\mathrm{D}\mathcal{R}_{\mathcal{A}}$ is a bounded operator, the error $u_{n+1}$ is small and we only need to make smaller and smaller corrections using  Lemma \ref{lemma2}.} This constant $K_4$ comes from writing a first order approximation of $\mathcal{R}_{\mathcal{U}}$ in an $\epsilon_1$-neighbourhood of $0$.

\vspace{3mm}

\noindent By initially choosing $\epsilon$ such that $K_4 \epsilon \leq \epsilon_1$, we can apply Lemma \ref{lemma2} to get the existence of $v_{n+1}$ such that

$$\mathcal{R}^{n+1}(s, V_n +v_{n+1}, h) = (s'_{n+1}, 0, h'_{n+1})$$ with $v_{n+1}$ satisfying the following

\begin{itemize}

\item $||v_{n+1}|| \leq \lambda_2^{-(n+1)}\epsilon$;

\item for all $k\leq n+1$, $d_{\eta} \big( \mathcal{R}^k(s, V_n + v_{n+1},p)) -  \mathcal{R}^k(s, V_n, p)) \big) \leq  \lambda_2^{k-(n+1)}\epsilon$.

\end{itemize}

\noindent It follows that $V_{n+1} = V_n + v_{n+1}$ satisfies the induction hypothesis.

\vspace{5mm}

\noindent Finally we set

$$ V(s,h) = \sum_{n=1}^{\infty}{v_{n}(s,h)}.$$

\noindent Since $v_{n}(s,h)$ depends continuously upon the variable $(s,h)$ (this is given by Lemma \ref{lemma2}) and since the series defining $V(s,h)$ converges uniformly, we can conclude that $\phi$ is a continuous function satisfying the conclusion of the theorem.

\end{proof}

\section{Convergence of renormalisations}

This section is dedicated to proving that elements belonging to the space defined by Theorem \ref{prestable} have successive renormalisation actually converging exponentially fast to $T_0$. Recall that {\bf we have made the  assumption that} $r = 3$. Define 

$$ \mathcal{K} := \{ \text{graph of} \ V \}$$ which is a codimension $d-1 + g-1$ submanifold of $\mathcal{U}$.

\noindent There is just a natural obstruction for this to happen that we have to take care of. Note that the function 

$$ T \longmapsto \int_0^1{\eta_T} $$ is invariant under $\mathcal{R}$ and vanishes for IETs (and for AIETs as well). Define 

$$ \mathcal{U}_0 = \{ T \in \mathcal{U} \ | \ \int_0^1{\eta_T} = 0  \} $$ 

\noindent and $$ \mathcal{K}_0 = \mathcal{U}_0 \cap \mathcal{K} $$ which is a codimension $d-1 + g-1$ submanifold of  $\mathcal{U}_0$ (this is easily seen as  $\int_0^1{\eta_T} $ only depends upon the coordinate in $\mathcal{P}$).  In this section we prove the following theorem

\begin{thm}

\label{fastconvergence}
Up to reducing $\mathcal{U}$ further the following hold true.
\begin{enumerate}

\item There exists a constant $\rho_1 <1$ such that for all $T$ in $\mathcal{K}$ there exists $C_T$

$$ d_{\mathcal{C}^1}(\mathcal{R}^n(T), \mathcal{M}) \leq C_T \rho_1^n$$

\item There exists a constant $\rho_2 <1$ such that for all $T$ in $\mathcal{K}_0$ there exists $D_T$

$$ d_{\mathcal{C}^1}(\mathcal{R}^n(T), T_0) \leq D_T \rho_2^n $$

\end{enumerate}

\end{thm}

\noindent Note that we are working with $\mathcal{C}^3$-GIET and that we ultimately obtain results of convergence with respect to the $\mathcal{C}^1$-norm.

\subsection{Size of dynamical partitions}

We introduce for a given $T$

$$ \Delta_n = \sup_{I \in \mathcal{P}_n}{|I|} $$ which we call the \textit{size of the dynamical partition} $\mathcal{P}_n$. We prove the following statement

\begin{proposition}
\label{dynamicalpartition}
There exists $\alpha<1$ such that for all $T \in \mathcal{K}$ there exists $L_T$ such that 

$$ \Delta_n \leq L_T \cdot \alpha^n $$

\end{proposition}

\noindent This statement is a rather easy consequence of Theorem \ref{prestable}. The fact that it holds true is a key fact that will allow us to derive fast convergence of iterated renormalisations to Moebius IETs for elements of $\mathcal{K}$, and to $T_0$ for elements of $\mathcal{K}_0$. 

\begin{proof}

\noindent Because $T$ is close to $T_0$, there is $\beta < 1$ such that $\Delta_1 < \beta \Delta_0$. Now we show that so long as $\mathcal{R}^nT$ remains in a vicinity of $T_0$ there exists $\alpha <1$ such that 

$$ \Delta_{n+1} \leq \alpha \cdot \Delta_n.$$   $\mathcal{R}^nT$ is defined to be the first return map of $\mathcal{T}$ on a certain interval $[0,x_n]$. Recall that $I_n^1, \cdots, I_n^d$ are continuity intervals of $\mathcal{R}^nT$. Let $\tilde{\Delta}_{n+1}$ the supremum of the lengths of the iterated images of the $I_{n+1}^1, \cdots, I_{n+1}^d$ by $\mathcal{R}^nT$ before they come back to $[0, x_{n+1}]$.  Because $\mathcal{R}^nT$ is close to $T_0$, we have 

$$ \tilde{\Delta}_{n+1} \leq \beta \sup_j{|I^j_n|}.$$ These images form a partition of $[0,x_n]$ and the partition $\mathcal{P}_{n+1}$ is obtained by propagating this partition using $T$ until it comes back to $[0,x_n]$. In turn, by applying Lemma \ref{bound1} we get that restricted to $I^j_n$, the iteration $T^k$ of $T$ have uniformly bounded distortion. It means that the subdivision of each element of $\mathcal{P}_n$ that defines $\mathcal{P}_{n+1}$ is uniformly smaller, namely that the length of each element of $\mathcal{P}_{n+1}$  is less that $\alpha$ times the length of the element of $\mathcal{P}_n$ in which it is contained for an $\alpha< 1$.  This proves that

$$ \Delta_{n+1} < \alpha \cdot \Delta_n.$$

\end{proof}

\subsection{Fast convergence to projective IETs}
\label{fastconvergenceprojective}

When one can prove a control of the size of the dynamical partition as in Proposition \ref{dynamicalpartition}, it is a well-known fact that iterated renormalisations converge in $\mathcal{C}^1$-norm to the \textit{projective} or \textit{Moebius} IETs. The group $\mathrm{PSL}(2,\mathbb{R})$ acts projectively by analytic diffeomorphisms on $\mathbb{RP}^1 = \mathbb{R} \cup \{ \infty \}$, a \textit{projective} or \textit{Moëbius} map is any restriction of such a map to an interval $I \subset \mathcal{R}$. A generalised interval exchange transformation is said to be \textit{projective} or \textit{Moëbius} (PIET) if the projection on the coordinate $\mathcal{P} = \mathrm{Diff^r_+([0,1])}$ consists of projective diffeomorphisms of $[0,1]$.

\vspace{2mm} \noindent This part is very classical, we are only going to quickly brush over the standard arguments which allow to prove this fast convergence. We follow the elegant proof due to Khanin and Teplinsky. In \cite{KhaninTeplinsky}, the authors introduce what they call the distortion of a diffeomorphism $f$ of the interval which encodes how cross-ratios are modified under the action of $f$. This distortion behaves nicely under compositions and it is easy to show using Lemma 6 in \cite{KhaninTeplinsky} that the $\log$ of distortion of (each branch of) $\mathcal{R}^nT$ is proportional to $\Delta_n$ The distortion of a map is close to $1$ if and only if it is $\mathcal{C}^0$-close to a Moëbius map. 

\vspace{2mm} \noindent Because we have proved in Propositon \ref{dynamicalpartition} that for $T \in \mathcal{K}$, $\Delta_n$ converges exponentially fast to $0$ with respect to the $\mathcal{C}^0$-norm. Because of the $\mathcal{C}^2$-bounds, this implies fast convergence with respect to the $\mathcal{C}^1$-norm and therefore we get the first part of Theorem \ref{fastconvergence}.

\subsection{Fast convergence to AIETs}

\noindent We begin to show that this fast convergence to AIETs occurs for PIETs.

\begin{proposition}
\label{projconvergence}
Let $T$ be a PIET belonging to $\mathcal{X}_0$. Then there exists a constant $\mu_1 < 1$ such that 

$$  \mathrm{d}_1(\mathcal{R}T, \mathcal{A}) \leq \mu_1 \cdot \mathrm{d}_1(T, \mathcal{A}) $$

\end{proposition}

\begin{proof}
We first remark that a projective diffeomorphism of $[0,1]$ is entirely determined by the integral of its non-linearity. We also have the following chain rule for the non-linearity

$$ \eta_{f \circ g}(x) = \mathrm{D}f(x) \cdot \eta_f(g(x)) + \eta_g(x).$$ We deduce from this formula that if $f$ is a diffeomorphism $J \longrightarrow K$ and a diffeomorphism $I \longrightarrow J$ we have 

$$ \int_{I}{\eta_{f \circ g}} = \int_J{\eta_f} + \int_I{\eta_{g}}.$$ We apply this fact to the dynamical partition induced by $\mathcal{R}T$. Recall that $I_0^1, \cdots, I_0^d$ are the intervals of continuity of $T$ and $\mathcal{P}^1_1, \cdots, \mathcal{P}_1^d$ the dynamical partition associated with $\mathcal{R}T$. For each branch $\phi_j \in \mathrm{Diff}^3_+([0,1])$ of $\mathcal{R}T$, according to the chain rule for the non-linearity, we have for all $1 \leq j \leq d$

$$ \int_{I_j^1}{\eta_{\phi_j}} = \int_{\mathcal{P}^j_1}{\eta_T}.$$ If we take $T$ in a sufficiently small neighbourhood of $T_0$ we can impose that there exists $c>0$ such that 

$$\frac{| \mathcal{P}^j_1 \cap I_k|}{|I_k|} > c  $$ for any $j,k$. This is derived form the fact that for any periodic (linear) IET, any $\mathcal{P}_1^j$ intersects any $I_k$ non-trivially and the continuity in $T$ of the dynamical partition. The hypothesis  $T \in \mathcal{X}_0$ is equivalent to 

$$ \int_0^1{\eta_T} =  \int_{I_0^1}{\eta_{\varphi_1}} + \cdots + \int_{I_0^d}{\eta_{\varphi_d}} = 0 $$ where the $\varphi_i$s are the branches of $T$. Since the $\varphi_i$s are projective, the $\eta_{\varphi_i}$s are of constant sign. We get that $\sup_j{|\int{\eta_{\phi_i}}|} \leq \mu_1 \sup_j{|\int{\eta_{\varphi_i}}|}$ because the $\int_{I_0^i}{\eta_{\phi_i}}$s are a obtained by subdividing "in a balanced way" the $\int_{I_0^i}{\eta_{\varphi_i}}$s and rearranging so each $\int_{I_0^i}{\eta_{\phi_i}}$ is the sum of subparts of each of the $\int_{I_0^i}{\eta_{\varphi_i}}$. By taking a sufficiently small neighbourhood of the identity in $\mathrm{Diff}^r_+([0,1])$ intersected with projective maps, we can make the norm $ f \mapsto |\int{\eta_f} |$ and the $\mathcal{C}^1$-norm(precisely the $\mathcal{C}^1$-norm of the difference with the identity map) as close as we like, which gives the result.

\end{proof}

\begin{proposition}
\label{affconvergence}

Let $T \in \mathcal{K}_0$. Then there exists $C'_T>0$ and $\mu_2<T$ such that 

 $$ \mathrm{d}_1(\mathcal{R}^nT, \mathcal{A}) \leq C'_T \cdot \mu_2^n.$$ 

\end{proposition}

\begin{proof}
$\mathcal{R}$ is $K$-Lipschitz with respect to the $\mathcal{C}_1$-norm in a neighbourhood of $T_0$ for a certain $K>0$. Up to restricting $\mathcal{K}_0$ to this neighbourhood we can assume that $\mathcal{R}$ is $K$-Lipschitz. Let $P$ be a PIET realising $\mathrm{d}_1(T, \mathcal{P})$.  We have that 

$$ \mathrm{d}_1(\mathcal{R}T, \mathcal{A}) \leq K \mathrm{d}_1(T, \mathcal{P})  + \mathrm{d}_1(\mathcal{R}P, \mathcal{A}).$$  Applying to $\mathcal{R}^nT$ we get 
 
$$ \mathrm{d}_1(\mathcal{R}^{n+1}T, \mathcal{A}) \leq K \mathrm{d}_1(\mathcal{R}^nT, \mathcal{P})  + \mathrm{d}_1(\mathcal{R}P_n, \mathcal{A}).$$  for $P_n$ realising  $\mathrm{d}_1(\mathcal{R}^nT, \mathcal{P})$. Using estimate of Proposition \ref{projconvergence} and fast convergence to projective maps we get

$$ \mathrm{d}_1(\mathcal{R}^{n+1}T, \mathcal{A}) \leq K C_T \rho_1^n + \mu_1\mathrm{d}_1(P_n, \mathcal{A})$$ where $P_n$ is the PIET realising $\mathrm{d}_1(\mathcal{R}^nT, \mathcal{P})$. We then have $\mathrm{d}_1(P_n, \mathcal{A}) \leq \mathrm{d}_1(\mathcal{R}^nT, \mathcal{A}) + \mathrm{d}_1(\mathcal{R}^nT, \mathcal{P}) \leq \mathrm{d}_1(\mathcal{R}^nT, \mathcal{A})  + C_T \rho_1^n$. We thus get

$$ \mathrm{d}_1(\mathcal{R}^{n+1}T, \mathcal{A}) \leq (K+\mu_1) C_T \rho_1^n + \mu_1\mathrm{d}_1(\mathcal{R}^nT, \mathcal{A}).$$ This is easily shown to imply the existence of $C'_T$ and $\mu_2 <1$ such that the proposition holds true.

\end{proof}

\subsection{Fast convergence to the fixed point}

\noindent We conclude by explaining how Proposition \ref{affconvergence} implies the second part of Theorem \ref{fastconvergence}. An element in $\mathcal{A} \cap \mathcal{K}_0$ is exactly an element of the stable space of $\mathcal{R}$ at $T_0$. We can use a reasoning analogous to that of the proof of Proposition \ref{affconvergence} to show that an element of $\mathcal{K}_0$ is exponentially close to the stable space of $\mathcal{R}$ restricted to $\mathcal{A}$. This implies that iterated renormalisations of $T_0$ converge exponentially fast to $T_0$. 

\vspace{2mm} Let $p_0 = (s_0, u_0, h_0)$ an element of $\mathcal{K}_0$ and let $p_n = (s_n, u_n, h_n)$ be $\mathcal{R}^n(p_0)$.  We know by Proposition \ref{affconvergence} that 

$$ ||h_n||_{\mathcal{C}^1} \longrightarrow 0 $$ exponentially fast \textit{i.e.} there exists $C(p_0) > 0$ and $\mu_2<1$ such that $||h_n||_{\mathcal{C}^1} \leq C(p_0)  \mu_2^n$. Recall that $\mathcal{R}_{\mathcal{U}}$ is differentiable in a neighbourhood of $T_0 = (0,0,0)$ and we have 

$$ \mathcal{R}_{\mathcal{U}}(u,s,h) = D_{\mathcal{U}} \mathcal{R}_{\mathcal{U}}(u) + D_{\mathcal{\mathcal{P}}} \mathcal{R}_{\mathcal{U}}(h) + o(||s|| + ||u|| + ||h||_{\mathcal{C}^1})$$ and 
$$ \mathcal{R}_{\mathcal{S}}(u,s,h) = D_{\mathcal{S}} \mathcal{R}_{\mathcal{S}}(s) + D_{\mathcal{\mathcal{P}}} \mathcal{R}_{\mathcal{S}}(h) + o(||s|| + ||u|| + ||h||_{\mathcal{C}^1}).$$ In particular we can derive that (up to restricting the neighbourhood of $T_0$ we are working with), 

$$ || \mathcal{R}_{\mathcal{U}}(u,s,h) || \geq \lambda^{-1} ||u|| - K ||h||_{\mathcal{C}^1} - \epsilon ||s|| $$  and 

$$ || \mathcal{R}_{\mathcal{S}}(u,s,h) || \leq \lambda ||s|| +K ||h||_{\mathcal{C}^1} +\epsilon ||u|| $$ 
for an arbitrarily small $\epsilon$, a constant $K > 0$ and a certain $\lambda <1$. Assume there exists $n_0$ such that  $||u_{n_0}||$ is significantly larger than both $||h_{n_0}||$ and $ ||s_{n_0}||$. Formally, assume the existence of constants $K_1$ such that

\begin{itemize}
\item $||u_{n_0}|| \geq K_1 C(p_0)\mu_2^n \geq K_1 ||h_{n_0}||$;

\item $||u_{n_0}|| \geq  ||s_{n_0}||$.

\end{itemize}

\noindent If $K_1$ is chosen sufficiently large and $\epsilon$ sufficiently small, this property holds for all $n \geq n_0$, in practice $K_1, \epsilon$ such that $\lambda + \epsilon + \frac{K}{K_1} \geq 1$ suffice.In this case one can show by induction that 

$$ \forall n \geq n_0, \ ||u_{n+1}|| \geq (\lambda^{-1} - \epsilon -\frac{K}{K_1}) ||u_n||.$$ Up to modifying $K_1$ and $\epsilon$ further so that $ (\lambda^{-1} - \epsilon -\frac{K}{K_1}) > 1$ we get that the sequence $||u_n||$ increases at a geometric rate. In particular it implies that $\mathcal{R}^n(p_0)$ leaves the neighbourhood $\mathcal{W}'$ which is a contradiction. We can therefore assume that there exists $K_1$ such that 

$$ \forall n, \ ||u_n|| \leq \max(||s_n||,K_1||h_n||) \leq ||s_n|| + K_1||h_n||.$$ Now we have 

$$ ||s_{n+1}||  +  K_1||h_{n+1}|| \leq \lambda||s_n|| + \epsilon ||u_n|| + K ||h_n|| \leq (\lambda + \epsilon) (||s_n|| + K_1 ||h_n||) + K'\mu_2^n $$ where $K'$ is another constant. From this inequality (and because $\lambda+ \epsilon < 1$) one finds that $||s_{n+1}||  +  K_1||h_{n+1}|| $ decreases at an exponential rate which implies the second part of Theorem \ref{fastconvergence}.

\section{Rigidity theorems}
\label{rigidity}

\noindent In this section we show how the fast convergence theorem (Theorem \ref{fastconvergence}) implies $\mathcal{C}^1$-conjugacy for elements of $\mathcal{K}_0$ which is Theorem \ref{maintheorm} and how this $\mathcal{C}^1$-conjugation can be improved to $\mathcal{C}^{1+\delta}$ using a method that was first used for the rigidity of critical circle mappings (see \cite{deFariadeMelo}).

\subsection{$\mathcal{C}^1$-rigidity}

\vspace{3mm} \noindent Consider $T$ a GIET belonging to $\mathcal{K}_0$. It is infinitely renormalisable, and displays the same combinatorics as that of $T_0$. It is classical this in that case $T$ is semi-conjugate to $T_0$ (we refer to \cite{Yoccoz2}, Proposition 7). By a theorem of Masur and Veech, a periodic interval exchange transformation is always uniquely ergodic and its unique invariant measure is the Lebesgue measure. In turn, $T$ is also uniquely ergodic. We are interested in the case where $T$ is conjugate via a $\mathcal{C}^1$ diffeomorphism of $[0,1]$ to $T_0$. In this case, the image of the Lebesgue measure by the $\mathcal{C}^1$ conjugacy is a measure of the form

$$ \mu(x)dx$$ where $dx$ denotes the Lebesgue measure and $h$ is a continuous positive function. This measure is in this case the unique invariant measure of $T$. Conversely, if $T$ preserves a measure of this form, it is $\mathcal{C}^1$-conjugate to $T_0$. The invariance of such a measure is equivalent to the following equation

\begin{equation}
\forall x, \ \mu(T(x)) = \frac{1}{\mathrm{D}T(x)}\mu(x).
\end{equation}

\noindent Our approach is to construct $h$ building upon the following remark: the equation above is equivalent to the following cohomological equation

\begin{equation}
\label{cohomological}
\log \mu \circ T  - \log \mu = - \log \mathrm{D}T  
\end{equation}

\noindent It is a standard fact (often referred to as Gottschalk-Hedlund
theorem) that if $U : X \longrightarrow X$ is a minimal homeomorphism of a compact space $X$, the equation above as a solution if and only if the Birkhoff sums of $\mathrm{D}T$ are uniformly bounded. Unfortunately, $T$ is not a homeomorphism of $[0,1]$ since it has discontinuity points. However, Marmi-Moussa-Yoccoz \cite{MMY1}  have shown that an equivalent statement still holds for minimal GIETs. 

\begin{lemma}[Marmi-Moussa-Yoccoz, \cite{MMY1}, Corollary 3.6]
\label{MarmiYoccoz}
Let $T$ be a minimal GIET without connections. Let $\varphi : [0,1] \longrightarrow \mathbb{R}$ be a function which is continuous on continuity intervals of $T$. Assume that Birkhoff sums of $\varphi$ are uniformly bounded. Then there exists a continuous $\phi : [0,1] \longrightarrow \mathbb{R}$ such that 

$$ \phi \circ T - \phi = \varphi.$$

\end{lemma}

\noindent We will now move on to proving that Birkhoff sums of $\log \mathrm{D}T$ are uniformly bounded. This statement is equivalent to the following proposition.

\begin{proposition}
\label{derivativebounded}
Assume $T \in \mathcal{K}_0$. There exists $F_T>1$ such that for all $x$ and for all $n\in \mathbb{N}$

$$ F_T^{-1} < \mathrm{D}(T^n)(x) < F_T.$$

\end{proposition}

\begin{proof}

The proof of this proposition relies on classical estimates of Birkhoff sums, via interpolating using special times corresponding to first returns of the induction. Precisely, we utilise to following fact: for any $x \in [0,1]$ and $n\in \mathbb{N}$ there exists integers $a_0, a_1, \cdots, a_k$ all smaller than a uniform constant $M$ (which can be taken as the larger first-return time used to define $\mathcal{R}T_0$) such that 

$$ T^n(x) = (\mathcal{R}^kT)^{a_k} \circ (\mathcal{R}^{k-1}T)^{a_{k-1}} \circ \cdots (\mathcal{R}T)^{a_1} \circ T^{a_0}(x).$$ Using the chain rule and passing to the logarithm gives 

$$ \log \mathrm{D}(T^n)(x) = \sum_{i=0}^k{\log \mathrm{D}(\mathcal{R}^kT)^{a_k})( x_i)  }$$ where $x_i =(\mathcal{R}^iT)^{a_k}\circ (\mathcal{R}^{i-1}T)^{a_{i-1}} \circ \cdots (\mathcal{R}T)^{a_1} \circ T^{a_0}(x)$. We get

$$ |\log \mathrm{D}(T^n)(x) | \leq \sum_{i=0}^k{ a_k || \log \mathrm{D}(\mathcal{R}^kT)||_{\infty}  }.$$ But $\mathcal{R}^kT$ converges exponentially fast to $T_0$ in $\mathcal{C}^1$-norm, and by concavity of the $\log$ fuction we get that $|| \log \mathrm{D}(\mathcal{R}^kT)||_{\infty} \leq D_T \cdot \rho_2^n$ which implies that for all $x$ and all $n\in \mathbb{N}$ 

$$ |\log \mathrm{D}(T^n)(x) | \leq M \sum_{i=0}^k{ \rho_2^n} $$ and this concludes the proof of the proposition.

\end{proof}

\noindent We easily deduce from Proposition \ref{derivativebounded} that Birkhoff sums of the function $\log \mathrm{D}T$ are uniformly bounded. Consequently, according to Lemma \ref{MarmiYoccoz}, there exists a positive continuous function $\mu$ such that 

$$ \log \mu \circ T  - \log \mu = - \log \mathrm{D}T.$$ The measure 

$$ \mu(x)dx $$ is absolutely continuous with continuous, never vanishing density. Thus $T$ is $\mathcal{C}^1$-conjugate to $T_0$.

\vspace{2mm} \noindent We conclude this section by remarking that the above discussion implies an estimate on the $\mathcal{C}^1$-norm of the conjugating map.

\begin{proposition}
\label{control1} 
There exists $K>0$ such that for any $T \in \mathcal{K}_0$, if we denote by $\varphi$ the map conjugating $T$ to $T_0$ we have

$$ || \mathrm{Id} - \varphi_T ||_{\mathcal{C}^1} \leq K \cdot d(T,T_0)$$  where $d$ is the $\mathcal{C}^2$ distance on $\mathrm{\chi}$.
\end{proposition}

\begin{proof}

This is a consequence of the fact the $\mathcal{C}^0$-norm of the solution to the cohomological equation in Lemma \ref{MarmiYoccoz} depend linearly on the supremum of the $\mathcal{C}^0$-norm of the Birkhoff sums.

\end{proof}

\subsection{Improving $\mathcal{C}^1$-conjugation to $\mathcal{C}^{1+\delta}$}

\noindent In this paragraph we point out a result coming from the thoery of  renormalisation of (critical) circle diffeomorphisms used in \cite{deFariadeMelo} which allows for an improvement of the regularity of the conjugating map in Theorem \ref{maintheorm}.

\begin{thm}
\label{deltareg}
There exists $\delta > 0$ such that the following holds. Let $T_1$ and $T_2$ be two elements of $\mathcal{K}_0$. The map that conjugates $T_1$ to $T_2$ is of class $C^{1+\delta}$, and this map converges to the identity if $T_2$ converges to $T_1$ in the $C^{1+\delta}$-topology.
\end{thm}

\noindent The construction by de Faria-de Melo \cite{deFariadeMelo} is explained in Appendix \ref{c1alpha}. This result actually allows for an improvement of the regularity of the manifold $\mathcal{K}_0$.

\section{$\mathcal{C}^1$-regularity of $\mathcal{K}_0$}

We show in this section how results of Marmi and Yoccoz on the cohomological equation actually imply that the function $V$ constructed in Section \ref{corrections} is of class $\mathcal{C}^1$.

\subsection{Precise description of the tangent space at $T_0$}

In Section \ref{affine}, we discussed coordinates on the space of affine interval exchange transformations that we recall here. The space $\mathcal{A}$ of affine interval exchange transformations (AIETs) on $d$ intervals with fixed combinatorics identifies with 

$$ \big\{  \lambda = (\lambda_i)_{1\leq i \leq d}; \mu = (\mu_i)_{1\leq i \leq d} \ | \ \sum_i{\lambda_i} = \sum_{i}{e^{\mu_i} \lambda_i} = 1  \big\}. $$

\noindent In this identification, $T_0$ corresponds to the coordinates $(\lambda^0, 0) = \big( (\lambda_i^0), (\mu_i^0 = 0) \big)$ and the space of (standard/linear) interval exchange transformations (IETs) to the subset $\{ \mu = \vec{0} \}$.

\vspace{2mm} \noindent The tangent space at $T_0$ of $\mathcal{A}$ canonically identifies with 

$$ T_p\mathcal{A} =\big\{ \lambda = (\lambda_i)_{1\leq i \leq d}; \mu = (\mu_i)_{1\leq i \leq d} \ | \ \sum_i{\lambda_i} =  \sum_{i}{\mu_i \lambda_i^0} = 0  \big\}.$$

\noindent The matrix of $\mathrm{D}_{T_0}\mathcal{R}$ at $T_0$ in these coordinates is of the form 

\[
M =
\left[
\begin{array}{c|c}
Q & * \\
\hline
0& A
\end{array}
\right]
\]

\noindent where $Q$ is the matrix of the restriction of $D\mathcal{R}$ to the tangent space to the subspace of IETs and $A$ is the matrix introduced in \ref{intersectionmatrix}. As we have already seen 

\begin{itemize}
\item $Q$ is an expanding matrix \textit{i.e.} there exists a norm on $\mathbb{R}^{d-1}$ and $\alpha > 1$ such that for all $v \in \mathbb{R}^{d-1}$, $ || Q\cdot v || \geq \alpha ||v||$.

\item $A$ is a hyperbolic matrix with exactly $g-1$ eigenvalues strictly larger than $1$.

\end{itemize}

\noindent Recall that we denote by $\mathcal{U}$ (respectively $\mathcal{S}$) the unstable (respectively stable) space of $D\mathcal{R}$. Note that the subspace generated by the coordinates $(\lambda_i)$ belongs to $\mathcal{U}$. 
\noindent Denote by $\mathcal{U}'$ the subspace generated by the coordinates $(\lambda_i)$ and the unstable space of the matrix $A$ acting only on the coordinate $(\mu_i)$, and let $\mathcal{S}'$ be subspace of $ T_p\mathcal{A}$ generated by the stable space of the matrix $A$ acting only on the coordinate $(\mu_i)$.  $\mathcal{U}'$  and  $\mathcal{S}'$ are not exactly the unstable and stable spaces of $D\mathcal{R}$ but $\mathcal{S}'$ satisfies the following.

\begin{proposition}
\label{decomposition}
The knowledge of both the coordinate of a vector in $\mathcal{U}'$ (with respect to the decomposition $T_p\mathcal{A} = \mathcal{U}' \oplus \mathcal{S}' $) and its coordinate in $\mathcal{S}$ (with respect to the decomposition $T_p\mathcal{A} = \mathcal{U} \oplus \mathcal{S} $) completely determines it.

\end{proposition}

\begin{proof}
This is a simple consequence of the fact a matrix taking the decomposition $T_p\mathcal{A} = \mathcal{U} \oplus \mathcal{S} $ to $T_p\mathcal{A} = \mathcal{U}' \oplus \mathcal{S}' $ is going to be triangular by blocks, by definition of $\mathcal{U}'$ and $\mathcal{S}'$.

\end{proof}

\subsection{Marmi-Yoccoz's work on the cohomological equation}

In the series of articles\cite{MMY1},  \cite{MMY} and \cite{MarmiYoccoz} the authors give an analysis of the cohomological equation over Roth type interval exchange transformations. It is important to point out that periodic interval exchange maps considered in this article satisfy their Roth type condition.

\noindent There are many results of interest contained in the aforementioned articles and we think it is fair to say that their most precise versions are contained in the most recent one \cite{MarmiYoccoz}. Consider $T_0$ a standard interval exchange map which is of \textbf{Roth type}. The definition of Roth type was first given in \cite{MMY} and generalises the notion of Roth type rotation number. We will not recall this definition here as it would require the introduction of too much material. It will be enough for our purpose to know that periodic interval exchange maps as considered in this article are of Roth type. 

\vspace{2mm}

\noindent In the coordinates introduced in the previous paragraph, the coordinates $\mu = (\mu_i)$ naturally identifies with the space of piecewise constant functions (constant on the intervals of continuity of $T_0$) whose average vanish. We denote by 

\begin{itemize}

\item $\mathcal{C}^{\delta}([0,1])$ the space of real-valued observables which are continuous everywhere, which are $\delta$-Hölder and whose average vanish;

\item $\mathcal{C}^{1+\delta}_m([0,1])$ the space of real-valued observables which are continuous on intervals of continuity of $T_0$, which are of class $\mathcal{C}^{1+\delta}$ restricted to those continuity intervals and whose average vanish;

\item $\Gamma_u$ the unstable space of the matrix $A$ (thought of as acting on the space of piecewise constant observables);

\item$ \Gamma_s$ the stable space of the matrix $A$.

\end{itemize}

\noindent We state Marmi-Yoccoz's result specified to our context (see \cite{MarmiYoccoz}, p127 Theorem 3.7 and Theorem 3.10 therein).

\begin{thm}[Marmi-Yoccoz, \cite{MarmiYoccoz}]
\label{MYthm}
Let $T_0$ be a linear interval exchange transformation of Roth type.  Let $\delta >0$. There exist two bounded linear operators $L_0 : \mathcal{C}^{1+\alpha}_m([0,1]) \longrightarrow \mathcal{C}^{\delta}([0,1])$ and $L_1 : \mathcal{C}^{1+\delta}_m([0,1]) \longrightarrow \Gamma_u$ such that for all $\varphi \in \mathcal{C}^{1+\alpha}_m([0,1])$ we have 

$$ \varphi = L_1(\varphi) + L_0(\varphi) \circ T_0 - L_0(\varphi).$$

\end{thm}

\noindent This theorem says in substance that every sufficiently  regular observable of mean zero is cohomologous to an essentially unique piecewise constant function. The fact that this result is true in regularity $\mathcal{C}^{1+\delta}$ is going to prove crucial in the proof of the regularity of the submanifold $\mathcal{K}_0$.

\subsection{Differentiability at $T_0$}
\label{diff0}

In this paragraph we first show that the manifold $\mathcal{K}_0$ has a tangent space at the point $T_0$. The proof actually works for any point of $\mathcal{K}_0$, but in order to lighten notation we first carry it out in this particular case. Recall that we have built in Section \ref{corrections} a map

$$ V : \mathcal{S} \times \mathcal{P} \longrightarrow \mathcal{U}$$ such that the interval exchange map $T_{s,h}$ of coordinates $(s,h,V(s,h))$ is $\mathcal{C}^{1+\delta}$-conjugate to $T_0$. We denote by $\varphi_{s,h}$ the map conjugating $T_{s,h}$ to $T_0$. The $\mathcal{C}^{1+\delta}$-distance of $\varphi_{s,h}$ goes to zero as $(s,h)$ goes to $(0, 0)$ (Theorem \ref{deltareg}).
\vspace{2mm}
By suitably rescaling $\mathrm{Id} + h$, one can think of its derivative as a piecewise continuous map whose continuity intervals are exactly that of $T_{s,h}$. We also denote by $\mu(h,s) = (\mu_i(h,s))_{i \leq d}$ the multiplier of the affine interval exchange map that is the shape of $T_{h,s}$. By definition, we have 

$$ \mathrm{D}T_{h,s} =  e^{\mu} \cdot(1 + \mathrm{D}h)  $$ where we think of $\mu$ as a piecewise constant map of the intervals of continuity of $T_{h,s}$. We also have the equation 

$$  \varphi_{h,s} \circ T_0 =  T_{h,s}  \circ \varphi_{h,s}.$$ Differentiating and taking the logarithm we get 

$$ (\log \mathrm{D} \varphi_{h,s}) \circ T_0 - \log \mathrm{D} \varphi_{h,s} = (\log \mathrm{D}(T_{h,s}) )  \circ \varphi_{h,s} $$ and thus 

$$ (\log \mathrm{D} \varphi_{h,s}) \circ T_0 - \log \mathrm{D} \varphi_{h,s} = \mu(h,s) + (\log(1 + \mathrm{D}h)) \circ \varphi_{h,s} .$$ Note that this equation resembles a lot that of Theorem \ref{MYthm}. The only difference is that there is no reason that $\mu(h,s)$ be an element of $\Gamma_u$. This is not too much of a problem for our purpose, as $s$ is fixed the projection of $\mu(h,s)$ onto $\Gamma_u$ completely determines $\mu(h,s)$ by Proposition \ref{decomposition}. To be perfectly rigorous, we can rewrite it 

$$ (\log \mathrm{D} \varphi_{h,s}) \circ T_0 - \log \mathrm{D} \varphi_{h,s} = \mu(h,s) - s + s + (\log(1+ \mathrm{D}h)) \circ \varphi_{h,s} $$ and thus find that 

$$ \mu(h,s) -s  = L_1\big( s + (\log(1 + \mathrm{D}h)) \circ \varphi_{h,s} \big) $$ and $$ \log \mathrm{D} \varphi_{h,s} = L_0 \big( s + (\log (1+ \mathrm{D}h)) \circ \varphi_{h,s} \big).$$

We have the following easy Proposition 

\begin{proposition}
If both $\mu(h,s) -s  $ and $ \varphi_{h,s}$ are depend $\mathcal{C}^1$ on $(h,s)$ then $V$ is of class $\mathcal{C}^1$. 
 
 \end{proposition}

\begin{proof}

$V(h,s)$ is completely determined by the datum of $(\lambda_i)_{i \leq d}$ and $(\mu_i)$. But $\lambda = (\lambda_i)$ are the lengths of the images by $\varphi_{h,s}$ of the intervals of continuity of $T_0$.
\end{proof}

\noindent The rest of the Section is dedicated to proving that  $\mu(h,s) -s  $ and $ \varphi_{h,s}$ are of class $\mathcal{C}^1$. Because $ \mu(h,s) -s  = L_1\big( s + (\log(1 + \mathrm{D}h)) \circ \varphi_{h,s} \big) $ and $ \log \mathrm{D} \varphi_{h,s} = L_0 \big( s + (\log(1 + \mathrm{D}h)) \circ \varphi_{h,s} \big)$. Because both $L_0$ and $L_1$ are bounded operators, it suffices it to show that $||( \log(1 + \mathrm{D}h)) \circ \varphi_{h,s}  - \log (1 + \mathrm{D}h) ||_{1+ \alpha} = o(||h||_3 + ||s||)$. 

\vspace{3mm}

\vspace{3mm} \noindent We now prove this inequality. To simplify notation, we put $g =  \log(1+ \mathrm{D}h)$ which by hypothesis is of class $\mathcal{C}^2$ and which is such that $ ||g||_{2} \rightarrow 0$ when $h \rightarrow_{\mathcal{C}^{3}} \mathrm{Id} $.

\vspace{3mm}

\paragraph{\bf Control of the $\mathcal{C}^0$-norm} We compute 

$$g \circ \varphi_{h,s}(x) - g(x) = g(x + \varphi_{h,s}(x) - x) -g(x) =  g'(x)(\varphi_{h,s}(x)-x) + o(\varphi_{h,s}(x)-x).$$

\noindent Since $||g'||$ is of the order of $||h - \mathrm{Id}||_3)$ and that $||\varphi_{h,s} - \mathrm{Id}||_0 = o(||h||_3 + ||s||)$ we get that $|g \circ \varphi_{h,s}(x) - g(x) | = o(||h||_3 + ||s||)$.

\vspace{3mm}

\paragraph{\bf Control of the $\mathcal{C}^{1+\delta}$-norm} \noindent First we recall the following general facts about $\delta$-Hölder functions. 

\begin{lemma}
\label{holder}
Assume $u$ and $v$ are two $\delta$-Hölder functions. We have 

\begin{enumerate}

\item $|| u \cdot v ||_{\delta} \leq ||u||_0 \cdot ||v||_{\delta} +  ||v||_0 \cdot ||u||_{\delta}$;

\item if $u$ is $\mathcal{C}^1$, we have $|| u \circ v || \leq  ||u'||_0 \cdot ||v ||_{\delta}$;
%

\item If $u$ is $\mathcal{C}^1$, we have $||u||_{\delta} \leq \max(||u||_0, 2^{1-\delta} ||u||_0^{1-\delta} ||u'||_0^{\delta})$.

\end{enumerate}

\end{lemma}

\begin{proof}
These facts are elementary and their proofs are left to the reader.
\end{proof}

\noindent Recall that $g$ is of class $\mathcal{C}^2$ and that by definition of $g$ and from Theorem \ref{deltareg} we have

\begin{itemize}

\item $||g||_0 = O(||h||_3)$;

\item  $||g'||_0 = O(||h||_3)$;

\item  $||g''||_0 = O(||h||_3)$;

\item $|| \varphi_{s,h} - \mathrm{Id} ||_1 = o(1)$;

\item $ || \varphi_{s,h}' - 1 ||_{\alpha} = o(1)$.

\end{itemize}

 \noindent We have 

$$ (g \circ \varphi - g)' = g'\circ  \varphi_{h,s} \cdot \varphi_{h,s}' - g' = (g'\circ \varphi_{h,s} - g') + g'\circ \varphi_{h,s}(\varphi_{h,s}' -1).$$

\noindent We first take care of the term $A = (g'\circ \varphi_{h,s} - g')$. $A$ is differentiable and $A' = g'' \circ \varphi_{h,s} \cdot \varphi_{h,s}' - g''$. In particular 

$$ || A' ||_0 \leq ||g''||_0 (1 + || \varphi_{h,s}'  ||_0) $$ and in a bounded neighbourhood of $(0,0)$ , $|| \varphi_{h,s}'  ||$ is uniformly bounded by a constant $K$. We can thus write $|| A' ||_0 = O(||h||_3)$.  \\ Then, for all $x$ we have 

$$ g' \circ \varphi_{h,s}(x) - g'(x) = \int_x^{\varphi_{h,s}(x) }{g''(t)dt} $$ from which we get $||A||_0 \leq ||g''||_0 \cdot || \varphi_{h,s} - \mathrm{Id} ||_0 $. We can thus write $||A||_0 = o(||h||_3 + ||s||)$. Applying the third point of Lemma \ref{holder} to $A$ we get 

$$ ||A||_{\delta} = o(||h||_3 + ||s||).$$

\vspace{2mm} \noindent We now deal with the second term $B = g'\circ \varphi_{h,s}(\varphi_{h,s}' -1)$. Using the first point of Lemma \ref{holder}, we get 

$$ || B ||_{\delta} \leq ||g'\circ \varphi_{h,s}||_0 \cdot ||\varphi_{h,s}' -1||_{\delta} + ||g'\circ \varphi_{h,s}||_{\delta} \cdot ||\varphi_{h,s}' -1||_{0}.$$ Recall that $||g'||_0 = O(||h||_3)$ and $||\varphi_{h,s}' -1||_{\delta} = o(1)$. In addition to that, we can apply the second point of Lemma \ref{holder} to get that $ ||g'\circ \varphi_{h,s}||_{\delta} \leq ||g''||_0 \cdot ||\varphi_{h,s}||_{\delta} $. Using the fact that $||\varphi_{h,s}||_{\delta}$ is uniformly bounded, we get 

$$ || B ||_{\delta} =o(||h||_3 + ||s||).$$ Altogether this implies that

$$ || (g \circ \varphi - g)' ||_{\delta} = o(||h||_3 + ||s||).$$ This terminates the proof of the following statement :

\begin{proposition}
The function $ V : \mathcal{S} \times \mathcal{P} \longrightarrow \mathcal{U}$ is differentiable at the point $(0, \mathrm{Id})$. 
\end{proposition}

\subsection{Differentiability at an arbitrary point of $\mathcal{K}_0$}

From this point it is not too difficult too derive the differentiability at any point of $K_0$. One can run the exact same argument as the one of \ref{diff0} at an arbitrary point using Theorem \ref{deltareg}. One will find that :

\begin{itemize}

\item the function $V$ is differentiable at any point of $\mathcal{K}_0$;

\item that its derivative can be expressed using the bounded operator $L_0$ and $L_1$ from Theorem \ref{MYthm};

\item the derivative at $(s,h)$ is actually the same as that at $T_0 =(0,0)$ up to a precomposition by $\varphi_{h,s}$; this implies that the derivative varies continuously with $(s,h)$ as $\varphi_{h,s}$ does.

\end{itemize}

\noindent This terminates the proof that the function $V$ is of class $\mathcal{C}^1$.

\appendix

\section{Properties of the renormalisation operator}
\label{appendix}

\subsection{The Banach structure on $\mathcal{X}^r$}
Let $r$ be an integer greater or equal to $1$. Recall that $\mathcal{X}^r_{\sigma} = \mathcal{X}^r$ is the space of GIETs with permutation $\sigma$ on $d$ intervals of class $\mathcal{C}^r$. In Section \ref{coordinates} we explained how $\mathcal{X}^r$ naturally identifies with 

$$ \mathcal{A} \times \mathcal{P} $$ where $\mathcal{A}$ is the space of affine IETs with permutation $\sigma$ and $\mathcal{P}$ is the product of $d$ copies of $\mathrm{Diff}^r_+([0,1])$ the set of orientation preserving $\mathcal{C}^r$ diffeomorphism of the interval. The set  $\mathrm{Diff}^r_+([0,1])$ can be seen as a subset of the vector space of real valued  $\mathcal{C}^r$-maps of the interval taking value $0$ in $0$ and value $1$ in $1$. The latter can be endowed with the $\mathcal{C}^r$-norm to give the structure of a Banach affine space modelled on the vector space $\mathcal{C}^r_0([0,1], \mathbb{R})$ of $\mathcal{C}^r$-maps vanishing at both $0$ and $1$. $\mathrm{Diff}^r_+([0,1])$ is easily seen to be an open subset of this Banach affine space with respect to the topology induced by the $\mathcal{C}^r$-norm, this naturally endows $\mathrm{Diff}^r_+([0,1])$ with the structure of a Banach manifold whose tangent space at any point naturally identifies with $\mathcal{C}^r_0([0,1], \mathbb{R})$.

\vspace{2mm} \noindent On the other hand, $\mathcal{A}$ naturally identifies with an open subset of the projective space $\mathbb{RP}^{2d-2}$ and by that mean is naturally endowed with a structure of finite dimensional smooth manifold which specialises into a structure of Banach manifolds. In turn, $ \mathcal{X}^r$ seen as the product $ \mathcal{A} \times \mathcal{P} $ is naturally endowed with the structure of a Banach manifold as a product of Banach manifold.

\subsection{An easy lemma on smooth functions}

\begin{lemma}
\label{lemmafunction1}
Let $I \subset \mathbb{R}$ be an open connected interval. The map

$$ \begin{array}{ccc}
\mathcal{C}^r(I, \mathbb{R}) \times I & \longrightarrow & \mathbb{R} \\
( \varphi, p) & \longmapsto & \varphi(p)
\end{array}$$ is of class $\mathcal{C}^1$.

\end{lemma}

\begin{proof}

We compute $$(\varphi +h)(p+\epsilon) = \varphi(p+\epsilon) + h(p+\epsilon) = \varphi(p) + \varphi'(p)\epsilon + o(\epsilon) + h(p) + h'(p)\epsilon + o(\epsilon). $$ But $h'(p)\epsilon $ is a $o(\sup(\epsilon, ||h||_{\mathcal{C}^1})$ therefore 
$$ \begin{array}{ccc}
\mathcal{C}^r(I, \mathbb{R}) \times I & \longrightarrow & \mathbb{R} \\
( \varphi, p) & \longmapsto & \varphi(p)
\end{array}$$ 
is of class $\mathcal{C}^1$ with derivative at $(\varphi,p)$ equal to 

$$  (h, \epsilon) \longmapsto  \varphi'(p)\epsilon + h(p) $$.

\end{proof}

\noindent An easy but important for our purpose consequence of this lemma is that if $f_1, \cdots, f_n$ are $\mathcal{C}^1$ maps then 

 $$ (f_1, \cdots, f_n, p) \mapsto f_n \circ f_{n-1} \circ \cdots \circ f_1 (p) $$ is of class $\mathbb{C}^1$, provided the for all $k$ the range of $f_k$ belongs to the interval of definition of $f_{k+1}$.

\subsection{Analytic properties of the renormalisation operator}

\noindent Recall the following definitions and notation from Section \ref{corrections}. We can identify a neighbourhood of $\mathcal{X}^r$ with an open neighbourhood of $0$ in the Banach space upon which $\mathcal{X} = \mathcal{A} \times \mathcal{P}$ is modelled. In these coordinates, we will use the notation $T_0 = (0_{\mathcal{A}},  0_{\mathcal{P}})$ where $0_{\mathcal{P}}$ represents the point $(\mathrm{Id}, \mathrm{Id}, \cdots,  \mathrm{Id}) \in \mathrm{Diff}^r_+([0,1])$.  Here $\mathcal{P}$ abusively denotes (a neighbourhood of $0$ in) the Banach space upon which $(\mathrm{Diff}^r_+([0,1]))^d$ is modelled. 
\noindent In these coordinates, we write 

$$ \mathcal{R} = (\mathcal{R}_{\mathcal{A}}, \mathcal{R}_{\mathcal{P}}).$$

\noindent Finally, we denote by $\pi_{\mathcal{A}}$ and $\pi_{\mathcal{P}}$ the projection from $\mathcal{X}$ onto $\mathcal{A}$ and $\mathcal{P}$ respectively.

\begin{proposition}

$\mathcal{R}$ is continuous in a neighbourhood of $T_0$ for the $\mathcal{C}^0$-topology.

\end{proposition}

\begin{proof}

This results from the continuity of the following functions, with respect to the $\mathcal{C}^0$-topology

\begin{enumerate}

\item restriction of a function to an interval;

\item evaluation of a function at a given point;

\item composition of functions.

\end{enumerate}

\end{proof}

\noindent We now move to proving that $\mathcal{R}_{\mathcal{A}}$ is differentiable. To achieve this we need a set of coordinates on $\mathcal{A}$. Recall that $\mathcal{A}$ is the set of affine interval exchange maps on $d$ intervals with permutation $d$. A point in $\mathcal{A}$ is completely determined by the its discontinuity points $0 < u^t_1 < \cdots < u^t_k < \cdots < u^t_{d-1} < 1$ at the "top" and their images $0 < u^b_1 < \cdots < u^b_k < \cdots < u^b_{d-1} < 1$ at the "bottom". These $2d-2$ parameters provide a set coordinates compatible with the smooth structure of $\mathcal{A}$.

\begin{proposition}
There exists a neighbourhood of $T_0$ in $\mathcal{X}^r$ such that $\mathcal{R}_{\mathcal{A}}$ is of class $\mathcal{C}^1$ in this neighbourhood for the $\mathcal{C}^r$-norm.
\end{proposition}

\begin{proof}

As indicated in the above discussion above, $\mathcal{R}_{\mathcal{A}}(T)$ is entirely determined by the positions of finitely many iterates of $T$ on finitely many points. We explain how positions can all be expressed as a finite combination of the functions from Lemma \ref{lemmafunction1} applied to coordinates of  $\pi_{\mathcal{A}}(T)= (u^t_1(T), \cdots, u^t_{d-1}(T), u^b_1(T), \cdots, u^b_{d-1}(T)  $ and $\pi_{\mathcal{P}}(T) = (\varphi_1(T), \cdots, \varphi_d(T))$ which will give the result.

\vspace{3mm} \noindent $\mathcal{R}(T)$ is the (rescaled) first return map of $T$ on an interval of the form $[0,T^k(u_i^t(T))]$ for a certain $k \in \mathbb{Z}$ and a certain $i \leq d-1$. Moreover, the discontinuities of $\mathcal{R}(T)$ are also of the form $T^k(u_i^t(T))$ and therefore it is enough to show that for any $k$ and $i$ there exists a neighbourhood of $T_0$ in $\mathcal{X}^r$ for which the function $T \mapsto T^k(u_i^t(T))$ is of class $\mathcal{C}^1$. Now denote by $l_k^t = u^t_{k} - u^t_{k-1}$ the length of the k-th interval of continuity of $T$ at the top and $l_k^t(T) = u^t_{k}(T) - u^t_{k-1}(T)$ the length of the k-th interval of continuity of $T$ at the bottom. These maps (depending upon $T$) are smooth. The restriction of $T$ to the interval $]u_{i-1}^t, u_i^t[$  is of the following form 

$$ x \mapsto l^b_j \cdot \varphi_i(\frac{x - u^t_{i-1}}{l_i}) + u^b_j $$ for a certain $j \leq d-1$.

\vspace{2mm} \noindent Now, $T^k(u_i^t(T))$ can be expressed as finitely many compositions of the function of Lemma \ref{lemmafunction1} applied to the $\varphi_i(T)$ and affine maps depending smoothly upon the $u_i^t(T)$s and $u_i^b(T)$s. This implies (by Lemma \ref{lemmafunction1})  that $T \mapsto T^k(u_i^t(T))$ is of class $\mathcal{C}^1$. This concludes the proof.

\end{proof}

\section{Fine grids and $\mathcal{C}^{1+\delta}$ homeomorphisms}
\label{c1alpha}

We reproduce here some material from \cite{deFariadeMelo} and apply it to the special case of periodic GIETs.

\subsection{Fine grids} A \textit{fine grid} is a sequence of finite partitions $(\mathcal{Q}_n)_{n\in \mathbb{N}}$ of $[0,1]$ such that \begin{itemize}

\item $\forall n \in \mathbb{N}$, $\mathcal{Q}_{n+1}$ is a refinement of $\mathcal{Q}_n$;

\item there exists an integer $a>0$ such that for all $n$, each atom of $\mathcal{Q}_n$ is the union of at most $a$ atoms of $\mathcal{Q}_{n+1}$;

\item there exists $c >0$ such that for all $I,J$ adjacent atoms of $\mathcal{Q}_n$ we have 

$$ c^{-1}|I| \leq |J| \leq c |I|.$$

\end{itemize}

\noindent One easily checks the following fact 

\begin{proposition}

Let $T$ be a GIET which is $\mathcal{C}^1$ conjugate to $T_0$. Then the dynamical partition $(\mathcal{P}_n)_{n\in \mathbb{N}}$ of $T$ form a fine grid.

\end{proposition}

\noindent The main technical tool of \cite{deFariadeMelo} is the following proposition

\begin{proposition}[de Faria-de Melo, \cite{deFariadeMelo}]
Let $h : [0,1] \longrightarrow [0,1]$ be a homeomorphism and assume that $(\mathcal{Q}_{n})$  is a fine grid. Assume furthermore that there exist positive constants $C>0$ and $\lambda<1$ such that for every $I,J$ adjacent atoms in $\mathcal{Q}_n$ we have 

$$ | \frac{I}{J} - \frac{h(I)}{h(J)} | \leq C \lambda^n.$$ Then 

\begin{enumerate}

\item There exists $\delta$ such that $h$ is of class $\mathcal{C}^{1+\delta}$.

\item $\sup_{x,y \in [0,1]}{\frac{|h'(x) - h'(y)|}{|x-y|^{\delta}}} \leq C$.

\end{enumerate}
\end{proposition}

\noindent This Proposition is not exactly stated as such in \cite{deFariadeMelo}: the second point is implicit and one will find it in the proof of Proposition 4.3, p358.

\subsection{Application to the conjugating map}

In this Section we apply the above material to our context. We prove the following

\begin{proposition}
\label{conjugation1}
There exists a uniform $\delta$ depending only on $T_0$ and a continuous positive function $A : (0,\nu) \rightarrow \mathbb{R}_+$ such that $\lim_{\epsilon \rightarrow 0}{A(\epsilon)} = 0$ such that the following holds. Assume $T_1$ and $T_2$ belong to $\mathcal{K}_0$. Then the map conjugating $T_1$ to $T_2$ is $A(d_{\mathcal{C}^1}(T_1,T_2))$-close to the identity in the $\mathcal{C}^{1+\delta}$.
\end{proposition}

\noindent The fact that the conjugating is $A(d_{\mathcal{C}^1}(T_1,T_2)$ close to the identity was already implicit in Section \ref{rigidity}. Indeed, we have the following fact :

\vspace{2mm} \noindent there exists $\kappa < 1$ such that  for $T_1$ and $T_2$  as in the Proposition above the following holds 

$$ d_{\mathcal{C}^1}(\mathcal{R}^n(T_1), R^n(T_2)) \leq A(d_{\mathcal{C}^1}(T_1,T_2)) \kappa^n.$$ for a certain function $A$ whose limit in $0$ is $0$. This is, as in the proof of Proposition \ref{control1}, because the $\mathcal{C}^0$-norm the solution to the cohomological equation depend linearly on that of the Birkhoff sums of the variable. In the case we are studying, the conjugating map is given by integrating the solution to the cohomological equation for the difference 

$$ \log \mathrm{D}T_1 - \log\mathrm{D}T_2.$$ Because $ d_{\mathcal{C}^1}(\mathcal{R}^n(T_1), R^n(T_2)) \leq A(d_{\mathcal{C}^1}(T_1,T_2)) \kappa^n$, Birkhoff sums of this difference are never any bigger that $D \cdot A(d_{\mathcal{C}^1}(T_1,T_2)) $ where $D$ is a uniform constant depending only on $T_0$. 

\vspace{2mm} \noindent Thus the only bit missing to prove Proposition \ref{conjugation1} is the fact that the derivative of the conjugating map is $\delta$-Hölder and that its Hölder-norm is controlled by  $A(d_{\mathcal{C}^1}(T_1,T_2))$. This will be a consequence of the following

\begin{proposition}
Let $T_1$ and $T_2$ be as in Proposition \ref{conjugation1}, let $h$ be the map conjugating $T_1$ to $T_2$ and let $(\mathcal{P}_n)$ be the sequence of dynamical partitions of $T_1$. Then there exists $\kappa'<1$ such that for all $n\in \mathbb{N}$  and adjacent $I,J$ in $\mathcal{P}_n$ we have 

$$ |\frac{I}{J} - \frac{h(I)}{h(J)}| \leq A(d_{\mathcal{C}^1}(T_1,T_2)) \cdot \kappa'^n.$$

\end{proposition}

\begin{proof}

We only give a sketch of the proof as it is already done in \cite{K1}[Section 9, p113-121] for circle diffeomorphisms with break points. The proof works exactly the same in this context. We describe below the main steps:
\begin{enumerate}

\item Because of the exponential convergence of renormalisations of $T_1$ and $T_2$, the estimate obviously holds for adjacent intervals that are in the base of the dynamical partition. This is just because $h$ maps the dynamical partition of $T_1$ to that of $T_2$ and that renormalisation converge exponentially fast at a rate depending only on $T_0$.

\item Now assume that $I$ and $J$ belong to the partition of level $m + n$. There exists an integer $k$ such that $T_1^k(I\cup J)$ belongs to a base interval for the partition $\mathcal{P}_m$.  This integer can be made small enough to guarantee that the measure of the union $\bigcup_{i\leq k}{ \mathcal{T}^i(I \cup J)}$ is of the order $\iota^n$ for a certain $\iota < 1$ depending on $T_0$.

\item Because iterated renormalisations of $T_1$ and $T_2$ converge very fast, if $m$ is taken sufficiently large then comparing $T^k(I)$ and $T^k(J)$ with their respective images in the base partition of $\mathcal{P}_{m+n}$ induces in a error that is exponentially small with $m$.

\item Next, the error induced when initially bringing back $I$ and $J$ to the base of $\mathbb{P}_m$ can be controlled by the fact that the distortion of $T^k$ is proportional to the measure of  $\bigcup_{i\leq k}{ \mathcal{T}^i(I \cup J)}$ (by applying the standard distortion Lemma \ref{bound1}) which is exponentially small. 

\item An appropriate choice of $p>0$ makes $m= p\times n$ big enough so that the control of the $d_{\mathcal{C}^1}(\mathcal{R}^m(T_1), \mathcal{R}^m(T_2) )$ is sufficient.

\end{enumerate}

\end{proof}

\bibliographystyle{plain}
	\bibliography{biblio.bib}

\end{document}